\documentclass[english,a4paper, abstract = true]{scrartcl}
\pdfoutput=1
\usepackage[applemac]{inputenc}

\usepackage{microtype}
\usepackage[shortlabels]{enumitem}
\usepackage[mathscr]{euscript}
\usepackage{a4wide,amsmath,amssymb,bm,graphicx,hyperref}
\usepackage[english]{babel}
\usepackage{amsthm}
\usepackage{thmtools, thm-restate}
\usepackage[T1]{fontenc}
\usepackage{dsfont}
\usepackage{mathtools}
\usepackage{tabularx}
\usepackage{parskip}

\usepackage{geometry}
\usepackage{verbatimbox}

\usepackage{marginnote}

\usepackage[dvips]{epsfig}
\usepackage[bf]{caption}

\usepackage{adjustbox}
\usepackage{dirtytalk}
\usepackage{mdframed}

\usepackage{tikz}
\usetikzlibrary{matrix}
\usetikzlibrary{arrows.meta}
\usetikzlibrary{cd}
\usetikzlibrary{calc}
\usetikzlibrary{hobby}

\newcommand{\sups}[1]{\ensuremath{^{\textrm{#1}}}}
\newcommand{\subs}[1]{\ensuremath{_{\textrm{#1}}}}

\newcommand{\bd}[1]{\boldsymbol{#1}}%
\newcommand{\mbf}[1]{\mathbf{#1}}%

\newcommand{\mtn}{\mathds N}%
\newcommand{\mtz}{\mathds Z}%
\newcommand{\mtr}{\mathds R}%
\DeclareMathOperator{\sign}{sign}
\DeclareMathOperator{\Vol}{Vol}
\DeclarePairedDelimiter\norm{\lVert}{\rVert}
\DeclarePairedDelimiter\abs{\lvert}{\rvert}
\DeclarePairedDelimiter\expect{\langle}{\rangle}

\makeatletter
\let\oldabs\abs
\def\abs{\@ifstar{\oldabs}{\oldabs*}}
\let\oldnorm\norm
\def\norm{\@ifstar{\oldnorm}{\oldnorm*}}

\let\oldexpect\expect
\def\expect{\@ifstar{\oldexpect}{\oldexpect*}}
\makeatother

\theoremstyle{plain}
\newtheorem{lemma}{Lemma}[section]
\newtheorem{corollary}[lemma]{Corollary}
\newtheorem{proposition}[lemma]{Proposition}

\newtheorem*{theorem}{Theorem}
\theoremstyle{remark}

\newtheorem*{remark}{Remark}
\theoremstyle{definition}
\newtheorem{defn}[lemma]{Definition}
\newtheorem*{example}{Example}

\newtheorem{manualtheoreminner}{Theorem}
\newenvironment{manualtheorem}[1]{%
  \renewcommand\themanualtheoreminner{#1}%
  \manualtheoreminner
}{\endmanualtheoreminner}

\newtheorem{manualpropinner}{Proposition}
\newenvironment{manualprop}[1]{%
  \renewcommand\themanualpropinner{#1}%
  \manualpropinner
}{\endmanualpropinner}

\newcommand{\defeq}{\vcentcolon=}

\usepackage{setspace}
\usepackage{cite}

\title{Closed embedded self-shrinkers of mean curvature flow}
\author{Oskar Riedler\footnote{Westfälische Wilhelms-Universität Münster, Mathematisches Institut, \href{mailto:oskar.riedler@uni-muenster.de}{oskar.riedler@uni-muenster.de}}}
\date{\normalsize \today}
\begin{document}
\maketitle

\begin{abstract}
In this article we show the existence of closed embedded self-shrinkers in $\mtr^{n+1}$ that are topologically of type $S^1\times M$, where $M\subset S^n$ is any isoparametric hypersurface in $S^n$ for which the multiplicities of the principle curvatures agree. This yields new examples of closed self-shrinkers, for example self-shrinkers of topological type $S^1\times S^k\times S^k\subset \mtr^{2k+2}$ for any $k$. If the number of distinct principle curvatures of $M$ is one the resulting self-shrinker is topologically $S^1\times S^{n-1}$ and the construction recovers Angenent's shrinking doughnut \cite{angenent-92}.
\end{abstract} 

\section{Introduction}
Consider a compact $n$-dimensional manifold $\Sigma$ that is smoothly immersed in $\mtr^{n+1}$ via a map $F_0:\Sigma\to \mtr^{n+1}$. A \textit{mean curvature flow of $F_0(\Sigma)$} is a family of smooth immersions $F_t: \Sigma\to\mtr^{n+1}$ where $t\in \mtr$ varies over some interval and for which
$$\partial_t F_t(x) = \bd{H}_t(x)$$
holds for all $t$. Here $\bd H_t(x)$ is the mean curvature of $F_t(\Sigma)$ at $F_t(x)$. In other words $F_t(\Sigma)$ flows along its mean curvature vector in $\mtr^{n+1}$. Due to compactness of $\Sigma$ such a flow necessarily becomes singular in finite time, see e.g. \cite{husken-monotonicity-90}.

By the work of Huisken \cite{husken-monotonicity-90}, Ilmanen \cite{ilmanen-95} and White \cite{white-94} rescaling $F_t(M)$ near the singular time in an appropriate way leads to weak limits that are so called \textit{self-shrinkers}, that is immersed manifolds whose mean curvature flow is given by dilations. These self-shrinkers then take a special role in the singularity theory of the mean curvature flow.

In this paper we use the theory of isoparametric foliations of the sphere $S^n$ to construct new examples of closed embedded self-shrinkers. Concretely we show:

\begin{manualtheorem}{A}\label{theorem}
For any isoparametric hypersurface $M$ in $S^n$, $n\geq 2$, for which the multiplicities $m_1$ and $m_2$ of the principal curvatures agree there is a closed embedded self-shrinker of topological type $S^1\times M$ in $\mtr^{n+1}$. This hypersurface is a union of homothetic copies of the leaves of the isoparametric foliation of $S^n$ associated to $M$.
\end{manualtheorem}

The theory of isoparametric hypersurfaces of the sphere $S^n$ is very rich, so the above theorem can be used to produce self-shrinkers of novel topology (for example $S^1\times S^{k}\times S^{k}\subset \mtr^{2k+2}$ for $k\in\mtn$ or $S^1\times SO(3)/(\mtz_2\times\mtz_2)\subset\mtr^5$). These hypersurfaces have previously been applied to the problem of mean curvature self-shrinkers by Chang and Spruck \cite{spruck-17}, who constructed for any isoparametric hypersurface $M$ of the sphere $S^n$ a self-shrinking end that is asymptotic to the cone $C(M)$.

The terminology of isoparametric hypersurfaces will be recalled in Section \ref{sec: iso}. The proof of Theorem \ref{theorem} works via a reduction of the shrinker condition to a geodesic equation in a two-dimensional manifold. Simple periodic solutions of this ordinary differential equation are then established by shooting methods very similar to \cite{angenent-92}, although the equation itself is quite different.

Denoting with $g$ the number of principal curvatures of a regular leaf of the isoparametric foliation one has in the case $g=1$ that the leaves become the latitudes of a sphere. The hypersurfaces found by Theorem \ref{theorem} are then rotationally invariant under the $O(n)$ action on $\mtr^{n+1}$ and topologically of type $S^1\times S^{n-1}$, the same topological type as the \say{shrinking donut} found by Angenent \cite{angenent-92}. It is currently an open question whether there exist embedded rotationally invariant self-shrinkers of type $S^1\times S^{n-1}$ in $\mtr^{n+1}$ other than Angenent's example, so we also remark:

\begin{manualprop}{B}\label{prop: g=1-angenent}
In the case $g=1$ the construction of Theorem \ref{theorem} gives Angenent's shrinking doughnut \cite{angenent-92}.
\end{manualprop}

The structure of this article is as follows: In section \ref{sec: iso+reduc+ode} we recall the necessary facts about isoparametric foliations, explain the reduction of the self-shrinker problem to an ordinary differential equation, and remark on some elementary properties of the resulting equation. In section \ref{sec: shooting} the shooting argument is presented and Theorem \ref{theorem} is shown with the exception of a technical proposition. Proposition \ref{prop: g=1-angenent} is also proved in section \ref{sec: shooting}. The aforementioned technical proposition is shown in section \ref{sec: proof}. 

The author would like to thank Peter McGrath for interesting and helpful discussions. The author gratefully acknowledges the
support of Germany's Excellence Strategy EXC 2044 390685587, Mathematics M\"unster: Dynamics-Geometry-Structure.

\section{Reduction and geodesic equation}\label{sec: iso+reduc+ode}

In subsection \ref{sec: iso} we recall some basic definitions and facts about isoparametric foliations. Subsection \ref{sec: reduc} explains the reduction procedure: a result due to Angenent \cite{angenent-92} gives that a hypersurface is a self-shrinker if and only if it is minimal in some metric $g\subs{Ang}$. The reduction theorem of Palais and Terng \cite{palais-terng-86} is then applied in order to reduce the shrinker property to a geodesic equation on an open subset of $\mtr^2$ equipped with a special metric. In subsection \ref{sec: geod} we present this geodesic equation and simplify its form.

\subsection{Isoparametric foliations on spheres}\label{sec: iso}

\begin{defn}
Let $M$ be a smooth Riemannian manifold. A smooth function $f:M\to\mtr$ is called \textit{isoparametric} if there are smooth functions $a, b: f(M)\to\mtr$ so that
\begin{align}
\|\nabla f\|^2 = a\circ f, \label{eq: iso-1}\\
\Delta f = b\circ f. \label{eq: iso-2}
\end{align}

\end{defn}
The geometric meaning of condition (\ref{eq: iso-1}) is that the fibres of $f$ form a (singular) transnormal system\footnote{Meaning if a geodesic is perpedicular to a leaf at any time that the geodesic then remains perpendicular to all leaves it intersects, cf. \cite{thorbergsson-foliations} for more details.}, in particular they are all equidistant to each other. Condition (\ref{eq: iso-2}) implies that the regular fibres of the foliation are of constant mean curvature in $M$ (cf. \cite{dominguez-notes}). If $M$ is a space-form one even has that the individual principal curvatures of such a fibre are constant along the fibre. Foliations that arise from the fibres of an isoparametric function are called \textit{isoparametric foliations}. A hypersurface is called an \textit{isoparametric hypersurface} if it is a regular leaf of an isoparametric foliation.

The classification of isoparametric foliations in spheres was initiated by Cartan in \cite{cartan-spheres-1, cartan-spheres-2, cartan-spheres-3}. This has proven to be a difficult problem and, despite a long and active history of research, it is in part still open. A significant part of the structure theory of these foliations was developed by Münzner in two seminal papers \cite{muenzner-1, muenzner-2}.

We review now some structural facts of isoparametric foliations of $S^n$, cf. \cite{muenzner-1, muenzner-2, wang-minimal, fkm, dominguez-notes} for proofs and further information:
\begin{enumerate}[label=(\roman*)]
\item The principal curvatures of any regular fibre are constant along the fibre.
\item The number of distinct principal curvatures of a regular fibre is the same for any two regular fibres. Denoting this number by $g$ one has that $g\in\{1,2,3,4,6\}$.
\item There are precisely two singular fibres $V_1$ and $V_2$. One has $\mathrm{dist}(V_1,V_2) = \frac\pi g$. These singular fibres are closed and minimal submanifolds of $S^n$.
\item Any regular fibre is of the form $M_\varphi \defeq \{ x\in S^n \mid \mathrm{dist}(x, V_1) = \varphi\}$, where $\varphi\in (0,\frac\pi g)$. These fibres are all diffeomorphic to oneanother.
\item The principal curvatures of a regular fibre $M_\varphi$ are of the form $\cot(\varphi), \cot(\varphi +\frac \pi g),..., \cot(\varphi+\frac{(g-1)\pi} g)$. Denoting with $m_1\defeq \mathrm{codim}(V_1)-1$, $m_2\defeq\mathrm{codim}(V_2)-1$ the multiplicities of these principle curvatures are $m_1, m_2, m_1, ...$ . In particular $n-1=\frac g2(m_1+m_2)$ and $m_1=m_2$ if $g$ is odd.
\item For $\varphi^*=\frac 2g\arctan(\sqrt{m_1/m_2})$ the hypersurface $M_{\varphi^*}$ is minimal in $S^n$.
\item The volume of a regular fibre is given by
\begin{equation}
\Vol(M_\varphi)= c \sin(\frac g2\varphi)^{m_1}\cos(\frac g2\varphi)^{m_2}\label{eq: vol}
\end{equation}
where $c$ a positive constant that does not vary in $\varphi$ (but will be different for different foliations).
\item There is a homogenous polynomial $F:\mtr^{n+1}\to\mtr$ (called the \textit{Cartan-Münzner polynomial}) of degree $g$ so that $F\lvert_{S^n}= \cos(g\varphi)$ and for which one has:
$$\|\nabla F(x)\|^2 = g^2 \|x\|^{2g-2},\qquad \Delta F(x) = \frac{g^2}{2}(m_1-m_2) \|x\|^{g-2}.$$
\end{enumerate}

\begin{example}The cases $g\in\{1,2,3\}$ were first classified by Cartan. The list of homogenous examples was completed by Takagi and Takahashi \cite{takagi-takahashi-72} based on previous work by Hsiang and Lawson \cite{hsiang-lawson-71}, here an example is called homogenous if the the fibres of the foliation arise as the orbits of an isometric action on $S^n$. The homogenous cases always arise as the principal orbit of the isotropy representation of a Riemannian symmetric space of rank 2, see \cite{dominguez-notes} for more detailed remarks and references also for the other cases.
\begin{enumerate}[label=(\roman*)]
\item When $g=1$ the isoparametric foliation is congruent (that is equal up to an isometric transformation) to the latitudes of the sphere $S^n$. One has $V_1 =\{ (1 , 0 , ...,0)\}$, $V_2=\{(-1,0,...,0)\}$ and $M_\varphi = \{\cos(\varphi)\}\times \sin(\varphi) S^{n-1} $ for $\varphi\in(0,\pi)$. These examples are homogenous and $m_1=m_2=n-1$. 
\item When $g=2$ the isoparametric foliation is congruent to the foliation by Clifford tori $S^{m_1}\times S^{m_2}$. One has $V_1  = \{ (0,...,0)\}\times S^{m_2}$, $V_2 = S^{m_1}\times (0,...,0)$ and $M_\varphi = \sin(\varphi) S^{m_1}\times \cos(\varphi) S^{m_2}$ for $\varphi\in(0,\frac\pi 2)$. The integers $m_1, m_2$ are arbitrary so long as $m_1+m_2=n-1$, in particular $m_1\neq m_2$ is possible. The fibres are the orbits of an isometric $O(m_1+1)\times O(m_2+1)$ action on $S^n$.
\item When $g=3$ one has $m_1=m_2 \in \{1,2,4,8\}$. The fibres of the foliation arise as the distance tubes of certain embeddings of the projective planes $\mathds{FP}^2$ ($=V_1$) in $S^{3m_1+1}$ where $\mathds F\in \{\mathds{R,C,H,O}\}$ is one of the real division algebras or the octonions. These examples are homogenous and the fibres are diffeomorphic to $SO(3)/(\mtz_2\times\mtz_2)$, $SU(2)/\mathds{T}^2$, $Sp(3)/Sp(1)^3$, $F_4/Spin(8)$ respectively.
\item For $g=4$ there is an infinite family, introduced by Ferus, Karcher, and Münzner in \cite{fkm}, which contains both homogenous as well as inhomogenous examples. Two additional homogenous cases beyond this family exist, else all examples belong to this family. Here $m_1\neq m_2$ is possible.
\item For $g=6$ one has $m_1=m_2\in\{1,2\}$, as was shown by Abresch \cite{abresch-83}. For both cases there exist homogenous examples. If $m_1=m_2=1$ it was shown by Dorfmeister and Neher \cite{dorfmeister-neher} that the homogenous example is the only one.
\end{enumerate}
\end{example}

\subsection{Reduction for self-shrinkers}\label{sec: reduc}

\begin{defn}\label{def: f-inv}
Let $F$ be the Cartan-Münzner polynomial of an isoparametric foliation of $S^n$. Define:
\begin{gather*}
\mathbf{f}:\mtr^{n+1}-\mtr_{\ge 0}\cdot (V_1\cup  V_2) \to (0,\infty)\times(0,\frac\pi g),\qquad x\mapsto \left(\|x\|, \frac{\arccos( F({x}/{\|x\|}))}{g} \right).
\end{gather*}
A set $X\subset\mtr^{n+1}$ is called \textit{$\mathbf f$-invariant} if there is a set $N\subset(0,\infty)\times(0,\frac\pi g)$ so that $X=\mathbf{ f}^{-1}(N)$. Compare with the notion of $F$-invariant in \cite{wang-minimal}.
\end{defn}

Note that the $\mathbf f$-invariant sets are precisely those sets that are unions of homothetic copies of the regular fibres of the isoparametric foliation - that is unions of sets of the form $r\cdot M_\varphi$ for $(r,\varphi)\in(0,\infty)\times (0,\frac\pi g)$. (Recall that for $\|x\|=1$ one has $F(x)=\cos(g\varphi)$, where $\varphi$ is the distance to the singular fibre $V_1$. So $\mbf f^{-1}(r,\varphi) =  \{x\in\mtr^{n+1}\mid \|x\|=r, \frac{x}{\|x\|}\in M_\varphi\}$.)

Recall (cf. \cite{angenent-92, drugan-survey}) that a closed submanifold $X\subset \mtr^{n+1}$ is a \textit{self-shrinker under mean curvature flow} (short: \textit{self-shrinker}) if and only if there is an $\tau>0$ such that $X$ is a minimal hypersurface in $\mtr^{n+1}$ equipped with the metric (which we refer to as the \textit{shrinker metric}):
$$g\subs{sh}=e^{-\frac{\tau\|x\|^2}{n}} \sum_{i=1}^n dx_i^2=e^{-\frac{\tau\|x\|^2}n}g\subs{Euc}.$$
The parameter $\tau$ is related to the extinction time of $X$. By rescaling $X$ if necessary we take $\tau=1$ in what follows.
\begin{proposition}Equipping $\mtr^{n+1}-\mtr_{\ge 0}\cdot (V_+\cup  V_-)$ with the shrinker metric $g\subs{sh}$ (with $\tau=1$) and $(0,\infty)\times(0,\frac\pi g)$ with the metric $g\subs{Subm}\defeq e^{-\frac{r^2}n} (dr^2+ r^2 d\varphi^2)$ one has:
\begin{enumerate}[label=(\roman*)]
\item $\mathbf f$ is a surjective and proper Riemannian submersion.
\item The mean curvature vector of a fiber $\mathbf f^{-1}(r,\varphi)$ is given by
$$e^{\frac{r^2}{2n}} \left((\frac1r-\frac rn)\nu_r + \frac{H(\varphi)}r\nu_\varphi\right),$$
where $H(\varphi)$ is the mean curvature of $M_\varphi\subset S^n$, $\nu_\varphi$ is the unit normal of $rM_\varphi$ in $rS^n$ equipped with $g\subs{sh}$, and $\nu_r$ is the unit normal of $rS^n$ in $\mtr^{n+1}$ equipped with $g\subs{sh}$.
\end{enumerate}
\end{proposition}
The proof is a standard calculation and from (ii) one sees that that the mean curvature of the fibres of $\mathbf f$ form a basic field of the Riemannian submersion, meaning that it is the horizontal lift of a vector field on the base manifold. Riemannian submersions with this property are the key ingredient in the reduction theory developed by Palais and Terng in \cite{palais-terng-86}, recall:
\begin{theorem}[Palais-Terng, cf. Theorem 4 in \cite{palais-terng-86}]
Let $\pi: (E, g_E)\to (B, g_B)$ be a Riemannian submersion for which the mean curvatures of the fibres form a basic field, then for a $k$-dimensional submanifold $X\subset B$ one has that $\pi^{-1}(X)$ is minimal in $E$ if and only if $X$ is minimal in $(B, V^{2/k}g_B)$. Here $V^{2/k}g_B$ is the metric given by
$$(V^{2/k}g_B)\,(b) = \Vol_{g_E}(\pi^{-1}(b))^{2/k}\, g_B(b).$$
\end{theorem}

Using (\ref{eq: vol}) one gets (up to a constant factor):
\begin{align*}
\Vol_{g\subs{sh}}(\mathbf f^{-1}(r,\varphi))^2 &= e^{-\frac{r^2(n-1)}n}\Vol_{g\subs{Euc}}(r\cdot M_\varphi)^2 = r^{2(n-1)}e^{-r^2(1-1/n)} \Vol_{S^n}(M_\varphi)^2\\
& =r^{2(n-1)}e^{-r^2(1-1/n)}\sin^{2m_1}(\frac g2\varphi)\cos^{2m_2}(\frac g2\varphi).
\end{align*}

The problem of finding $\mathbf f$-invariant hypersurfaces that are minimal with respect to $g\subs{sh}$ is then reduced to finding geodesic segments in $(0,\infty)\times(0,\frac\pi g)$ equipped with the metric
\begin{equation} r^{2n-2} e^{-r^2} \sin(\frac g2\varphi)^{2m_1}\cos(\frac g2\varphi)^{2m_2} (dr^2+r^2d\varphi^2).\label{eq: metric}\end{equation}
We conclude:
\begin{proposition}\label{prop: periodic-to-closed}
A $\mathbf f$-invariant hypersurface $X\subset \mtr^{n+1}$ is a closed self-shrinker in $\mtr^{n+1}$ if and only if $N\defeq \mathbf f(X)$ is a closed geodesic in $(0,\infty)\times(0,\frac\pi g)$ with respect to the metric (\ref{eq: metric}).
\end{proposition}
\subsection{Geodesic equation}\label{sec: geod}

For the metric (\ref{eq: metric}) one gets the following geodesic equation, where $\alpha$ denotes the angle between $\frac{dr}{dt}$ and $\frac{d\varphi}{dt}$:
\begin{align*}
\frac {dr}{dt}& = \cos\alpha\, G(r,\varphi), \\
\frac {d\varphi}{dt}& = \sin\alpha\, \frac{G(r,\varphi)}r, \\
\frac {d\alpha}{dt}& = \sin\alpha\,  r \, \partial_r \frac{G(r,\varphi)}r- \cos\alpha\,  \frac1{r} \, \partial_\varphi G(r,\varphi).
\end{align*}
Here
$$G(r,\varphi) = r^{-n+1}e^{r^2/2}\sin(\frac g2\varphi)^{-m_1}\cos(\frac g2\varphi)^{-m_2}.$$
Since we are not directly interested in the parametrisation of the geodesic but rather in its orbit we perform a substitution $\frac {dt\sups{new}}{dt\sups{old}} = \frac1rG(r,\varphi)$ to simplify the equation:
\begin{align*}
\frac {dr}{dt}&= r\cos\alpha, \\
\frac {d\varphi}{dt} &= \sin\alpha, \\
\frac {d\alpha}{dt} &= \sin\alpha\, (r^2-n)  + \frac g2 \cos\alpha\, (m_1\cot(\frac g2\varphi)-m_2 \tan(\frac g2\varphi)).
\end{align*}
We simplify once more by letting $\theta(t)\defeq \frac g2 \varphi(t)$, substituting $\frac {dt\sups{new}}{dt\sups{old}}= \frac g2\frac1{\sin(2\theta)}$, letting $\xi(t)\defeq \frac g2\ln(\sqrt{\frac2g} r(t))$ and $m\defeq \frac 2g n = m_1+m_2 +\frac2g$ to get:
\begin{align*}
\frac {d\xi}{dt} &= \cos\alpha\, \sin(2\theta), \\
\frac {d\theta}{dt} & = \sin\alpha\,\sin(2\theta), \tag{$*$}\\
\frac {d\alpha}{dt} &= \sin\alpha\,\sin(2\theta)(e^{\frac4g\xi}-m)+2\cos\alpha \,l(\theta).\end{align*}
Here $l(\theta)=m_1\cos^2(\theta)-m_2 \sin^2(\theta)$.

If $\theta'(t)\neq0$ then $\xi$ may be (locally) given the form of a graph over $\theta$. This graph obeys the following ODE:
\begin{equation}
\frac {d^2\xi}{d\theta^2}(\theta) = \frac{\xi''(t)}{\theta'(t)^2}- \frac{\xi'(t)}{\theta'(t)} \frac{\theta''(t)}{\theta'(t)^2}= - \left(1+ (\frac{d\xi}{d\theta})^2\right)\left(e^{\frac4g\xi}-m + 2 H(\theta)\frac{d\xi}{d\theta}\right). \tag{$**$}
\end{equation}
Here
$$H(\theta) = \frac{l(\theta)}{\sin(2\theta)} = \frac{m_1}2 \cot(\theta) - \frac{m_2}2\tan(\theta).$$

The ODE $(*)$ can be formulated for all initial conditions $(\xi,\theta,\alpha)\in\mtr^3$. But in coordinates $(\xi,\theta)$ the domain $(0,\infty)\times(0,\frac\pi g)$ has been transformed to $\mtr\times(0,\frac\pi2)$; so we are only interested in solutions where the $\xi$ and $\theta$ components remain in that domain. We set $\mathcal D\defeq \mtr\times(0,\frac\pi2)\times\mtr$.

The ODE $(*)$ admits two trivial families of solutions in $\mathcal D$, namely for any $k\in \mtz$:
\begin{gather*}
(\xi,\theta,\alpha)\,(t) = (\xi(0)+2(-1)^k\frac{\sqrt{m_1m_2}}{m_1+m_2}t,\,\arctan\left(\sqrt{\frac{m_1}{m_2}}\right), \pi k),\\
(\xi,\theta,\alpha)\,(t) = (\frac g4\ln m,\, \mathrm{arccot}\left(e^{(-1)^k 2t}\cot(\theta(0))\right), \frac\pi2+\pi k).
\end{gather*}
The first of these solutions lifts to the cone $\mtr_{>0}\cdot M_{\varphi^*}$ over the minimal hypersurface of the isoparametric folation, which is a minimal submanifold of $\mtr^{n+1}$. The second lifts to the round sphere, albeit with the singular fibres $V_1, V_2$ removed.

\subsection{Elementary properties of $(*)$ and symmetry}

We briefly note some elementary properties of solutions of $(*)$, proofs are standard and are thus omitted. 

\begin{proposition}
\begin{enumerate}[label={(\roman*)}]
\item For any $(\xi_0, \theta_0, \alpha_0)\in \mtr^3$ there is a unique solution $\gamma$ of $(*)$ with initial condition $\gamma(0)=(\xi_0, \theta_0, \alpha_0)$. This solution is smooth and has domain of definition all of $\mtr$, i.e. solutions exist for all times.
\item Suppose $(\xi_s,\theta_s,\alpha_s)$ converges in $\mtr^3$ to a point $(\xi_\infty,\theta_\infty,\alpha_\infty)$. Denote by $\gamma_s$ the solution of $(*)$ with initial condition $(\xi_s, \theta_s, \alpha_s)$ and $\gamma_\infty$ the solution of $(*)$ with initial condition $(\xi_\infty,\theta_\infty,\alpha_\infty)$. Then $\gamma_s$ converges uniformly on compacta to $\gamma_\infty$.
\item Solutions of $(*)$ with initial condition in $\mathcal D$ remain in $\mathcal D$ for all times.
\end{enumerate}
\end{proposition}

For our investigation we are interested in periodic solutions of $(*)$. These will be found with the help of a discrete symmetry of the ODE $(*)$.

\begin{defn}
Define $\theta^*\defeq\arctan(\sqrt{m_1/m_2})$ and let
$$S:\mtr^3\to\mtr^3, \qquad (\xi, \theta, \alpha)\mapsto (\xi, 2\theta^*-\theta, \pi-\alpha).$$
\end{defn} 
\begin{remark}
Note that $\theta^*$ is the solution in $(0,\frac\pi2)$ of $l(\theta)=0$. Additionally the map $S$ is an involution that reflects $\theta$ at $\theta^*$ while sending $\cos\alpha\to-\cos\alpha$ and $\sin\alpha\to\sin\alpha$. In the event that $m_1=m_2$ one has $\theta^*=\frac\pi4$ and $l(2\theta^*-\theta)=-l(\theta)$ as well as $\sin(2(2\theta^*-\theta))=\sin(2\theta)$. For $m_1=m_2$ one has $S(\mathcal D) = \mathcal D$.
\end{remark}

\begin{proposition}If $m_1=m_2$ then for any $x\in\mathcal D$ one has $S(\gamma_1(-t)) = \gamma_2(t)$ for all $t\in\mtr$, where $\gamma_1,\gamma_2$ are the solutions to $(*)$ with initial conditions $x$ and $S(x)$, respectively.
\end{proposition}

If $x=S(x)$ then $\gamma_1=\gamma_2$ in the above proposition and one gets $S(\gamma_1(t))=\gamma_1(-t)$. Noting that $(*)$ is further invariant under transformations of the form $\alpha\mapsto \alpha+2\pi k$ for $k\in\mtz$ then immediately gives a criterium for finding periodic solutions:

\begin{corollary}\label{cor: symm}
Let $m_1=m_2$ and $x\in\mathcal D$ with $S(x)=x$, let $\gamma$ be solution of $(*)$ with initial condition $x$. If there are $T\neq0$ and $k\in\mtz$ so that
$$S(\gamma(T))= \gamma(T)+\begin{pmatrix}0\\0\\2\pi k\end{pmatrix}$$
then the $\xi$ and $\theta$ components of $\gamma$ are periodic and $2T$ is a period.
\end{corollary}
Note that one has $S(\xi,\theta,\alpha)=(\xi,\theta,\alpha+2\pi k)$ for some $k\in\mtz$ if and only if $\theta=\theta^*$ and $\alpha=\frac\pi2 + \pi j$ for some $j\in\mtz$.

\section{Existence of periodic curves}\label{sec: shooting}

In light of Corollary \ref{cor: symm} we wish to find geodesic segments that begin and end on the line $\theta=\theta^*$, with both intersections being orthogonal. We begin with the following defintion:

\begin{defn}\label{def: type}
For $\xi_0\in\mtr$ let $(\theta_{\xi_0}(t), \xi_{\xi_0}(t), \alpha_{\xi_0}(t))$ denote the solution of $(*)$ with initial condition $\xi(0) = \xi_0, \theta(0)=\theta^*$ and $\alpha(0) = \frac\pi2$. Then:
\begin{enumerate}[label=(\roman*)]
\item $\xi_0$ is said to be of \textit{type 1} if there is a $T>0$ so that $\theta_{\xi_0}(T) = \theta^*$ and $\xi_{\xi_0}'(t) \neq 0$ for all $t\in (0,T)$.
\item $\xi_0$ is said to be of \textit{type 2} if there is a $T>0$ so that $\xi_{\xi_0}'(T) = 0$ and $\theta_{\xi_0}(t) \neq \theta^*$ for all $t\in (0,T)$.
\item $\xi_0$ is said to be of \textit{type 3} if $\xi_{\xi_0}'(t) \neq 0$ and $\theta_{\xi_0}(t) \neq \theta^*$ for all $t>0$.
\end{enumerate}
\end{defn}

Note that type 1 and type 2 are not exclusive, whereas a point is type $3$ precisely if it is not type 1 or type 2. In fact a point that is both of type 1 and and type 2 corresponds to a curve segment that orthogonally meets the $\theta = \theta^*$ line at its start and its end. If $m_1=m_2$ this leads to a solution for which the $\xi$ and $\theta$ components are periodic, which corresponds a closed embedded self-shrinker in $\mtr^{n+1}$ of topological type $S^1\times M$, here $M$ is diffeomorphic to the leaves of the isoparametric foliation. The following argument then finds a value $\xi_0^*$ that is both type 1 and type 2.
\begin{remark}
For $m_1=m_2$ one can see that $\xi_0 = \frac g4\ln m$ is the only type 3 point, as in this case type 3 points correspond to embedded mean-curvature convex self-shrinkers that are topologically a sphere. By \cite{husken-monotonicity-90} the only closed embedded mean-curvature convex self-shrinkers are round spheres, which in this setting are given by the line $\xi=\frac g4\ln m$.
\end{remark}

\begin{figure}[h]
\begin{center}
\begin{tikzpicture}
\draw[thick, blue] (-2,4.5) arc(90:0: 1.5cm and 1cm) coordinate(arc1);
\draw[thick, blue] (arc1) arc(0:-66: 2.5 cm and 1.5cm);
\draw[thick, green] (-2,3.5) arc(90:0: 1.8cm and 1.2cm) coordinate(arc2);
\draw[thick, green ] (arc2) arc(0:-90: 1.8 cm and 1.5cm);
\draw[thick, orange] (-2, 2.5) arc(90:0:2.2cm and 1.4 cm ) coordinate(arc3);
\draw[thick, orange] (arc3) arc(0:-118: 1.5cm and 1.7 cm);
\draw[thick, red] (-2, 1.5) to [curve through = { (-1,1.5) .. (2,.5) .. (2.5,.15)}] (3,-.1);
\draw[very thick, -latex] (-2, 1.2) -- (3.3,1.2);
\node[anchor = north] at (3.3,1.2) {$\theta$};
\draw[-latex, very thick] (-2,-1) -- (-2,5);
\node[anchor = east]  at (-2,5) {$\xi$};
\end{tikzpicture}

\caption{Examples of curves of different type. The blue curve is type 1 but not type 2, the green curve is both type 1 and type 2, the orange curve is type 2 but not type 1, the red curve is type 3.}
\label{fig: types}
\end{center}
\end{figure}
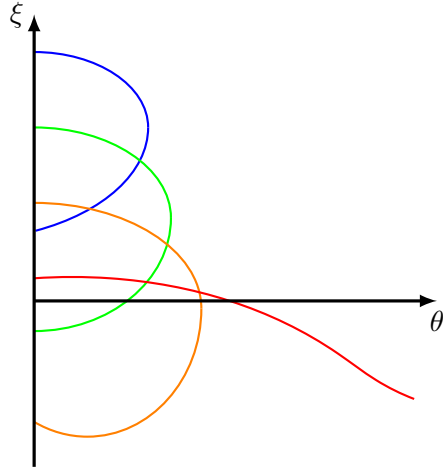

Define:
$$\xi_0^*= \inf \{r\in\mtr\mid \xi_0 \text{ is type 1 for all $\xi_0>r$ }\}.$$

\begin{proposition}\label{prop: position}We have:
\begin{enumerate}[label=(\roman*)]
\item $\xi_0^*<\infty$.
\item $\xi_0^* >\frac g4\ln m$.
\item $\xi_0^*$ is not type 3.
\end{enumerate}
\end{proposition}

This proposition will be proven in Section \ref{sec: proof}. For now we make use of the following elementary lemma:

\begin{lemma}\label{lemma: extrema}
For $(\xi,\theta)\,(t)\in\mathcal D$ one has the following characterisation of extrema:
\begin{enumerate}[label=(\roman*)]
\item If $\xi'(t) = 0$, then $\sign(\xi''(t)) = \sign(\frac g4\ln m-\xi(t))$.
\item If $\theta'(t) = 0$, then $\sign(\theta''(t)) = \sign(\theta^*- \theta)$.
\end{enumerate}
\end{lemma}
\begin{proof}
$\xi'(t)=0$ if and only if $\cos(\alpha(t))=0$, so
$$\xi''(t) = - \sin\alpha \sin(2\theta) \alpha'(t) = - \sin^2\alpha \sin^2(2\theta)(e^{\frac4g\xi}-m).$$
In the same way $\theta'(t)=0$ if and only if $\sin(\alpha(t))=0$, so
$$\theta''(t) = \cos\alpha\sin(2\theta)\alpha'(t) = 2\cos^2\alpha \sin(2\theta) l(\theta).$$
Now $l(\theta)>0$ for $\theta < \theta^*$ and $l(\theta)>0$ for $\theta >\theta^*$.
\end{proof}

This now gives:

\begin{proposition}\label{prop: type12}
$\xi_0^*$ is type 1 and type 2.
\end{proposition}
\begin{proof}
Since $\xi_0^*$ is not type 3 by Proposition \ref{prop: position} (iii) it must be at least one of type 1 or type 2. We first assume that $\xi_0^*$ is type 1 but not type 2, then we have a $T>0$ so that for all $\epsilon>0$ small enough one gets:
$$\cos(\alpha_{\xi_0^*}(t) ) \neq 0 \quad \forall t\in [\epsilon, T+\epsilon],\qquad \theta_{\xi_0^*}(T+\epsilon)< \theta^*.$$

Since solutions of the ODE $(*)$ vary continuously in the initial conditions (with respect to the topology of uniform convergence on compacta) one finds a neighbourhood $U(\epsilon)$ of $\xi_0^*$ so that for all $\xi_0\in U(\epsilon)$ one has
$$\cos(\alpha_{\xi_0}(t))\neq 0\quad \forall t\in [\epsilon, T+\epsilon],\qquad \theta_{\xi_0}(T+\epsilon)<\theta^*.$$
Additionally one has $|\frac d{dt}\xi_{\xi_0}(t) | ≤ 1$, whence if one chooses $\epsilon$ small enough there is a neighbourhood $V$ of $\xi_0^*>\frac g4\ln m$ with $\xi_{\xi_0}(t)>\frac g4\ln m$ for all $t\in [0,\epsilon]$ and $\xi_0\in V$. By Lemma \ref{lemma: extrema} $\xi(t)$ can only have maxima when $\xi(t)>\frac g4\ln m$, which implies that $\cos(\alpha_{\xi_0}(t))\neq 0$ for all $t\in(0,\epsilon]$ and $\xi_0\in V$.

The above shows that for $\xi_0\in V\cap U(\epsilon)$ one has that $\xi_0$ is type 1, contradicting the definition of $\xi_0^*$.

The assumption that $\xi_0^*$ is type 2 but not type 1 leads to a contradiction via similar argument, we carry this out:

Note first that $\xi_{\xi_0^*}''(0)<0$, so the next extremum must be a minimum (Lemma \ref{lemma: extrema} implies that whenever $\xi'(t)=0$ one has either a maximum, a minimum, or $\xi$ is the trivial solution $\xi=\frac g4 m$). This gives a $T>0$ such that for all $\epsilon>0$ small enough:
$$\theta_{\xi_0^*}(t) \neq \theta^*\quad \forall t\in[\epsilon, T+\epsilon] , \qquad \cos(\alpha_{\xi_0^*}(T+\epsilon)) >0$$
As before one gets a neighbourhood $U(\epsilon)$ of $\xi_0^*$ so that this extends to all initial conditions $\xi_0\in U(\epsilon)$, i.e. for all $\xi_0\in U(\epsilon)$:
$$\theta_{\xi_0}(t) \neq\theta^*\quad \forall t\in[\epsilon, T+\epsilon], \qquad \cos(\alpha_{\xi_0}(T+\epsilon)) >0$$
Since $\theta_{\xi_0}'(t)=1$ one again gets for $\epsilon$ small enough a neighbourhood $V$ of $\xi_0^*$ with $\theta_{\xi_0}(t)\neq\theta^*$ for all $t\in(0,\epsilon]$ and $\xi_0\in V$. This shows that all points $\xi_0\in V\cap U(\epsilon)$ are of type 2 but not type 1, again contradicting the definition of $\xi_0^*$.
\end{proof}

As discussed, this yields a periodic geodesic via Corollary \ref{cor: symm} in the case $m_1=m_2$.

\begin{lemma}
When $m_1=m_2$ one has that the periodic geodesic $t\mapsto (\xi_{\xi_0^*}, \theta_{\xi_0^*})\,(t)$ is simple.
\end{lemma}
\begin{proof}
Let $2T>0$ denote the period of the geodesic, then $\theta(0)=\theta(T)=\theta^*$. For $t\in(0,T)$ this gives $\xi_{\xi_0^*}'(t) <0$ and $\theta_{\xi_0^*}(t)>\theta^*$, similarly if $t\in (T,2T)$ then $\xi'_{\xi_0^*}(t)>0$ and $\theta_{\xi_0^*}(t)<\theta^*$. This immediately implies that the geodesic is simple.
\end{proof}

Together with Proposition \ref{prop: periodic-to-closed} this proves the main theorem of the paper:

\begin{manualtheorem}{A}
For any isoparametric hypersurface $M$ in $S^n$, $n\geq 2$, for which the multiplicities $m_1$ and $m_2$ of the principal curvatures agree there is a closed embedded self-shrinker of topological type $S^1\times M$ in $\mtr^{n+1}$. This hypersurface is a union of homothetic copies of the leaves of the isoparametric foliation of $S^n$ associated to $M$.
\end{manualtheorem}

In the case $m_1\neq m_2$ Proposition \ref{prop: type12} remains true and yields a simple geodesic segment that starts and ends on the $\theta=\theta^*$, meeting this line orthogonally in both places. In between the two ends one has $\theta>\theta^*$ and the same arguments give another geodesic arc with the same properties, except now $\theta<\theta^*$.

It may be useful to connect the end-points of these two arcs by line segments $\theta=\theta^*$ (which are geodesics). Doing so gives a simple closed curve consisting piecewise geodesic segments and having external angle sum equal to $0$. By the Theorem of Gauß-Bonnet this curve then encloses a total Gauß curvature of $2\pi$. Such a curve is an essential ingredient in \cite{drugan-variational}, where an adapted curve shortening flow is used to generated closed geodesics.

\subsection{The case studied by Angenent}

The case $g=1$ yields an embedded self-shrinker in $\mtr^{n+1}$ of topological type $S^1\times S^{n-1}$, which is invariant under an isometric $O(n)$ action on $\mtr^{n+1}$. This is the case investigated by Angenent in \cite{angenent-92}. In this subsection we relate Angenent's construction to ours and show that they give the same self-shrinker.

Let $\omega\defeq \{ (x_0,...,x_n) \in \mtr^{n+1}\mid \sum_{i=0}^nx_i^2=1, x_0=0\}$, $e_0\defeq (1,0,...,0)\in\mtr^{n+1}$. Then the construction of \cite{angenent-92} yields a self-shrinker of the form
\begin{equation}
\{x(t)\, e_0+\tilde r(t)\,\omega\mid t\in(a,b)\},\label{eq: ang-inv}
\end{equation}
where $x:(a,b)\to\mtr, \tilde r: (a,b)\to \mtr_{>0}$ are smooth functions and $a,b\in\mtr$. Note that sets the form (\ref{eq: ang-inv}) are precisely the $\mathbf f$-invariant sets from Definition \ref{def: f-inv}, where $\mathbf f$ arises from the isoparametric foliation of $S^n$ with $V_1= e_0$. To be more precise, if $r:(a,b)\to \mtr_{>0}$, $\varphi:(a,b)\to (0,\pi)$ one has:
\begin{equation}
\mathbf f^{-1}(\{(r(t),\varphi(t))\mid t\in (a,b)\}) = \{ r(t)\cos(\varphi(t))\,e_0+r(t)\sin(\varphi(t))\,\omega\mid t\in (a,b)\}.\label{eq: f-inverse}
\end{equation}
In \cite{angenent-92} the variational condition that a set of the form (\ref{eq: ang-inv}) is a self-shrinker is reduced to $\{(x(t), r(t))\mid t\in (a,b)\}$ being a geodesic segment in
\begin{equation}
(\mtr\times\mtr_{>0}, \tilde r^{2n-2}e^{-\tau(x^2+\tilde r^2)}(dx^2+d\tilde r^2)). \label{eq: metric-ang}
\end{equation}
Here $\tau>0$ is related to the extinction time and in \cite{angenent-92} one has $\tau=\frac14$. For ease of comparison we take $\tau=1$ and then up to a constant conformal factor the transformation implicit in (\ref{eq: f-inverse}) gives an isometry to the metric (\ref{eq: metric}) on $\mtr_{>0}\times (0,\pi)$, as is easy to see (recall $g=1$). It follows that any geodesic of (\ref{eq: metric-ang}) is a reparametrisation of a geodesic in (\ref{eq: metric}). Carrying out the additional coordinate changes of subsection \ref{sec: geod} one sees that the following map sends the orbits of solutions of $(*)$ to the orbits of geodesics of (\ref{eq: metric-ang}):
$$\Psi:\mtr\times(0,\frac\pi2)\to \mtr\times\mtr_{>0},\qquad (\xi,\theta)\mapsto (e^{\frac 2g\xi}\cos(2\theta), e^{\frac 2g\xi}\sin(2\theta)).$$
In particular $x(t)=0$ if and only if $\theta=\frac\pi4=\theta^*$.

Denote with $(\tilde r_R,x_R)\,(t)$ the evolution of a geodesic in (\ref{eq: metric-ang}) with initial conditions $(\tilde r(0), x(0))=(R,0)$ and $(\tilde r'(0), x'(0))=(0,1)$. Define:
$$R_* \defeq \inf\{\tilde R>0\mid\forall R>\tilde R: \exists t_1>0\text{ so that } x_R(t_1)=0\text{ and } \tilde r_R'(t)<0\ \forall t\in (0,t_1)\}.$$
Angenent then shows that the geodesic $(\tilde r_{R_*}, x_{R_*})$ meets the line $x=0$ orthogonally after a finite time. A symmetry argument as in Corollary \ref{cor: symm} then shows that this yields a simple periodic geodesic.

In order to show Proposition \ref{prop: g=1-angenent} we start with the following lemma. It follows from elementary arguments using continuity of solutions of the relevant ODEs in initial conditions, similar to Proposition \ref{prop: type12}.

\begin{lemma}\label{lemma: ang-switch}
For any neighbourhood $U$ of $R_*$ there are $R\in U$ and $t_1(R)>0$ so that
$$x_R(t_1(R))=0,\quad \tilde{r}_R'(t_1(R))>0,\quad \text{and }x(t)>0 \text{ for all }t\in (0,t_1(R)).$$
Similarly for any neighbourhood $V$ of $\xi_0^*$ there are $\xi_0\in V$ and $T(\xi_0)>0$ so that
$$\theta_{\xi_0}(T(\xi_0))=\theta^*,\quad \xi_{\xi_0}'(T(\xi_0))>0,\quad \text{and }\theta_{\xi_0}(t)>\theta^* \text{ for all }t\in(0,T(\xi_0)).$$
\end{lemma}

\begin{manualprop}{B}
In the case $g=1$ the construction of Theorem \ref{theorem} gives Angenent's shrinking doughnut \cite{angenent-92}.
\end{manualprop}
\begin{proof}
Let $L:\mtr\to\mtr$ be the reparametrisation of a geodesic so that $\Psi((\xi,\theta)\,(L(t))) = (x,\tilde r)\,(t)$. We take $L'(t)>0$ for all $t$. Then if $x(t)=0$ one has by an elementary calculation:
\begin{equation}
\sign(\xi'(L(t)))= \sign(\tilde r'(t))\label{eq: g=1-sign}
\end{equation}
Assume first $R_*> e^{\frac 2g\xi_0^*}$. Then apply the first part of Lemma \ref{lemma: ang-switch} to $R_*$ to get initial conditions $R$ arbitrarily close to $R_*$ and times $t_1(R)$ for which
$$x_R(t_1(R))=0,\quad \tilde r'_R(t_1)>0,\ \text{ and }\ x_R(t)>0 \text{ for all }t\in (0,t_1(R)).$$
This then transforms under $\Psi^{-1}$ to values $\xi_0$ close to $\frac g2\ln(R_*)$ (which is larger than $\xi_0^*$), by (\ref{eq: g=1-sign}) one then gets $\xi_{\xi_0}'(L(t_1)) >0$, while $\theta_{\xi_0}(L(t))>\theta^*$ for all $t\in(0,t_1)$.

These points are not type 1, contradicting the definition of $\xi_0^*$. The contradiction for $R_* < e^{\frac 2g\xi_0^*}$ is similar.
\end{proof}

\section{Proof of Propositon \ref{prop: position}}\label{sec: proof}

For the proof of Proposition \ref{prop: position} each of the points (i), (ii), and (iii) is considered separately in subsections \ref{subsec: (i)}, \ref{subsec: (ii)}, and \ref{subsec: (iii)} respectively. For the proof of these statements we also use two lemmas about the general dynamics of $(*)$, which are proven in subsection \ref{subsec: cross}. 

For a rough overview of the proof of points (i) and (ii), which are the more technical parts, see the beginnings of subsections \ref{subsec: (i)} and \ref{subsec: (ii)} as well as Figures \ref{fig: position(i)} and \ref{fig: position(ii)}.

\subsection{Crossings in finite time}\label{subsec: cross}

In this subsection we prove two useful lemmas that expand on the analysis of extrema in Lemma \ref{lemma: extrema}. The lemmas state that if $\xi'(t)$ points towards the $\frac g4 \ln m$ line then we reach this line in finite time, the same holding true for $\theta$ if $\theta'(t)$ points toward $\theta^*$. The proof of Proposition \ref{prop: position} (i) uses Lemma \ref{lemma: theta-turn} below, and Proposition \ref{prop: position} (iii) and Lemma \ref{lemma: theta-turn} use Lemma \ref{lemma: r-turn}.

\begin{lemma}\label{lemma: r-turn}
If for some $t_0\in\mtr$ one has $\xi(t_0)>\frac g4 \ln m$ ($\xi(t_0)<\frac g4 \ln m$) and $\xi'(t_0)<0$ ($\xi'(t_0)>0$) then there exists a time $T\in(0,\infty)$ so that $\xi(t_0+T)=\frac g4\ln m$.
\end{lemma}
\begin{proof}
If this were not true then $\xi(t)>\frac g4\ln m$ for all $t>t_0$. By Lemma \ref{lemma: extrema} we would then have that any extremum of $\xi$ is a maximum when $t>t_0$. Since $\xi'(t_0)<0$ it follows that $\xi$ has no extrema for times $>t_0$ and then $\xi(t)$ is monotonically decreasing in $t$ and bounded below by $\frac g4\ln m$ by assumption, hence it must converge.

Since $\xi(t)$ is bounded we find that $\alpha'(t)$ is bounded, whence $\xi'(t)= \cos\alpha\, \sin(2\theta)$ must converge to $0$ (since $\theta'(t)$ is also automatically bounded). So either $\lim_{t\to\infty}\alpha(t)\in \frac\pi2+\pi\mtz$ or $\lim_{t\to\infty}\theta(t)\in \{0,\frac\pi2\}$. We briefly show that the first condition implies the second and then continue working only with the second.

The condition $\alpha(t)\to\frac\pi2+k\pi$ for some $k\in\mtz$ gives $\sin\alpha\to(-1)^{k}$. For large times the dynamics of $\theta(t)$ are then given by
$$\theta'(t)=(-1)^k\sin(2\theta) + O(\cos\alpha).$$
This gives that $\theta(t)$ converges to either $0$ or $\frac\pi2$ as $t\to\infty$, depending on whether $k$ is even or odd.

Assuming now $\theta(t)\to0$ as $t\to\infty$ one gets from Lemma \ref{lemma: extrema} that $\theta(t)$ admits no extrema for $t$ large enough, so $\sin\alpha<0$ for $t$ large enough. This gives $\alpha(t) \in 2\pi\mtz +(\pi,\frac{3\pi}2)$ for $t$ large enough. 
 However for such $t$
$$\alpha'(t) = \sin\alpha \sin(2\theta)(e^{\frac4g\xi}-m) + 2\cos\alpha\, l(\theta)$$
is a sum of two strictly negative terms. $\alpha$ then decreases for large times, so there is an $\epsilon>0$ with $|\cos\alpha|>\epsilon$ for large enough $t$. In particular $\alpha'(t)<-2\epsilon l(\theta)$ for large $t$, where $l(\theta)$ converges to $m_1$. This means that in finite time $\alpha$ exits the interval $2\pi k + (\pi,\frac{3\pi}2)$ from the bottom, contradicting that $\sin\alpha<0$ for all $t$ large enough.

The case $\theta(t)\to\frac\pi2$ can be treated in the same way. This contradiction then implies the statement for $\xi(t_0)>\frac g4\ln m$ and $\xi'(t_0)<0$. The case $\xi(t_0)<\frac g4\ln m$ and $\xi'(t_0)>0$ is also completely analogous.\end{proof}

\begin{lemma}\label{lemma: theta-turn}
If for some $t_0$ one has $\theta(t_0)<\theta^*$ ($\theta(t_0)>\theta^*$) and $\theta'(t_0)>0$ ($\theta'(t_0)<0$) then there exists a time $T\in(0,\infty)$ so that $\theta(t_0+T)=\theta^*$.
\end{lemma}
\begin{proof}
As before assuming that $\theta(t)<\theta^*$ for all $t>t_0$ and $\theta'(t_0)>0$ implies that $\theta(t)$ has only minima as extrema, hence $\theta'(t)>0$ for all $t$ and $\theta(t)$ converges (being bounded from above by $\theta^*$).

We first assume that $e^{\frac 4g\xi(t)}$ remains bounded as $t\to\infty$, this implies that $\alpha'(t)$ remains bounded and then from convergence of $\theta$ one gets that $\theta'(t)=\sin\alpha\,\sin(2\theta)$ converges to $0$. Since $\sin(2 \theta)$ remains bounded away from $0$ one gets that $\alpha$ converges to some element of $\pi\mtz$. Since $\xi(t)$ is not allowed to go to $+\infty$ one gets that $\lim_{t\to\infty}\alpha(t)\in \pi+2\pi\mtz$.

Performing a coordinate transform $R=\ln\xi$ the system of ODEs $(*)$ becomes:
\begin{align*}
\theta'(t) &= \sin\alpha\,\sin(2\theta),\\
R'(t) &= \cos\alpha\,\sin(2\theta) R,\\
\alpha'(t) &= \sin\alpha\,\sin(2\theta)(R^{\frac4g}-m)+2\cos\alpha\,l(\theta).
\end{align*}

From $\alpha(t)\to\pi+2\pi k$ for some $k\in\mtz$ one gets $R(t)\to0$ and $\theta(t)\to\theta^*$. But the fixpoint $(\theta, R,\alpha)=(\theta^*, 0, \pi+2\pi k)$ is hyperbolic and at this point the above ODE has as linearization:

$$\frac d{dt}\begin{pmatrix}\theta \\ R \\ \alpha \end{pmatrix} \approx \begin{pmatrix}0& 0& -\sin(2\theta^*)\\ 0 & - \sin(2\theta^*) & 0 \\ -2 l'(\theta^*) & 0&\sin(2\theta^*)m \end{pmatrix}\begin{pmatrix}\theta-\theta^* \\ R \\ \alpha -(\pi+2\pi k)\end{pmatrix}.$$
The system then has a one dimensional stable manifold - this is the line
$$(\theta(t), R(t), \alpha(t)) = (\theta^*, R(0) \exp(-\sin(2\theta^*) t), \pi +2\pi k).$$
Since we are assuming $\theta(t)<\theta^*$ for all $t>t_0$ the solution cannot lie on the stable manifold, yielding a contradiction.

To complete the proof of the lemma we must show that $e^{\frac4g \xi(t)}$ cannot be unbounded under the hypothesis $\theta(t)<\theta^*$ for all $t>t_0$ and $\theta'(t_0)>0$. First recall the graph form $(**)$:
$$\frac{d^2\xi}{d\theta^2}(\theta)=-\left(1+(\frac{d\xi}{d\theta})^2\right)\left(e^{\frac 4g\xi}-m+2H(\theta)\frac{d\xi}{d\theta}\right).$$
Whence if $\theta<\theta^*$, $\xi>\frac g4\ln m$ and $\frac{d\xi}{d\theta}>0$ one has $\frac{d^2\xi}{d\theta^2}<0$, even becoming arbitrarily negative if $\xi$ becomes arbitrarily large. So if $\xi(t)$ is unbounded from above it cannot eventually be monotonic in $\theta$ (and hence in $t$ by $\theta'(t)>0$) and must admit maxima, in fact infintely many such maxima. Between two maxima there must be a minimum, which can only happen for values of $\xi(t)$ less than $\frac g4\ln m$.

However at each minimum one has $\theta'(t) = \sin(2\theta)$, which may be bounded from below since $\theta$ stays away from $\{0,\frac\pi2\}$. Since the system $(*)$ admits a Lipschitz constant on $\{(\theta,\xi,\alpha)\mid \xi < \frac g4\ln(m)+1\}$ one finds that at each minimum of $\xi$ the parameter $\theta$ increases by some positive number admitting a bound from below. This contradicts the assumption that $\theta(t)<\theta^*$ for all $t>t_0$.

The case $\theta(t_0)>\theta^*$, $\theta'(t_0)<0$ can be treated analogously.
\end{proof}

\subsection{Proof of Proposition \ref{prop: position} (i)}\label{subsec: (i)}

The proof of Proposition \ref{prop: position} (i) is divided into two parts. First we show that for $\xi_0$ large enough there is a $T_2>0$ so that the solution to $(*)$ with initial value $(\xi , \theta, \alpha)\ (t=0) = (\xi_0, \theta^*, \frac\pi2)$ has the property:
$$\theta'(T_2)=0,\quad 0<\theta(T_2)-\theta^*<\frac1{\xi_0},\quad\xi_0-\xi(T_2)<\frac1{\xi_0},\quad\text{while $\xi'(t)<0$ for all $t\in(0,T_2)$}.$$

So $\theta$ has an extremum at $T_2$, which by Lemma \ref{lemma: extrema} is a maximum and $\theta'(T_2+\epsilon)<0$ for small $\epsilon>0$. Then by Lemma \ref{lemma: theta-turn} one has that $\theta$ reaches $\theta^*$ in finite time and so $\xi_0$ cannot be of type 3. The proof then continues by contradiction, assuming that $\xi_0$ is not of type 1 means it must be of type 2. Being of type 2 means that $\xi$ must travel all the way to some value $<\frac g4 \ln(m)$ where we have an extremum of $\xi$ - all the while $\theta$ is not allowed to cross the line $\theta^*$.

The proof by contradiction is carried out in Lemma \ref{lemma: infty-xi0-finish}, here one assumes that conditions of this scenario have been set: there is some time $T_3$ at which $\xi(T_3)=\frac 4g\ln(m)$ all the while $\theta(t)>\theta^*$ and $\xi'(t)<0$ for $t\in (0,T_3]$. Using bounds for the value of $\theta(T_3)$ one is however able to show that even in this worst-case-scenario $\theta$ crosses the value $\theta^*$ before any extremum of $\xi$ is possible, contradicting the assumption that $\xi_0$ is type 2. Hence, since it cannot be type 3, it must have been type 1.

\begin{figure}
\begin{center}
\begin{tikzpicture}
\node (xi_top_up) at (0,1) {};
\node (xi_top_down) at (0,-1) {};
\node (curve_start) at (0,.5) {};
\node[draw, anchor = center, circle, fill = black, inner sep = 0pt, minimum width=3pt] (theta_extremum) at ($(curve_start)+(2,-1)$) {};
\node (xi_bot_up) at ($(xi_top_down)-(0,1)$){};
\node (xi_bot_down) at ($(xi_bot_up) - (0,4)$) {};
\node[anchor =   east] at (xi_bot_down) {\footnotesize $\theta=\theta^*$};
\node(xi_bot_mid_left) at ($(xi_bot_up) - (.5,.8)$) {};
\node[anchor =  north east] at (xi_bot_mid_left) {\footnotesize $\xi=\frac4g\ln m$};
\node (xi_bot_mid_right) at ($(xi_bot_mid_left) + (4,0)$) {};
\node (curve_cont) at ($(xi_bot_up)+(1.9,-.2)$) {};
\node (curve_cont2) at ($(curve_cont)$) {};
\draw[very thick,-latex] (xi_top_down) -- (xi_top_up) node[anchor = east] {$\xi$};
\draw[thick] (curve_start) arc(90:0: 2cm and 1cm);
\draw[dashed, thick, blue] (theta_extremum) arc(0:-8:2.5cm and 5cm);
\draw[loosely dotted, very thick] (xi_top_down) -- (xi_bot_up);
\draw [very thick] (xi_bot_up) -- (xi_bot_down) ;
\draw [very thick, -latex] (xi_bot_mid_left) -- (xi_bot_mid_right) node[anchor = south]{$\theta$};
\draw [dashed, thick, blue] (curve_cont) arc(-10:-40: 2cm and 5cm) coordinate(end);
\draw [dashed, thick, blue] (end) arc(-40:-130: 0.5cm and 1.5cm);
\draw[thick, red] (end) arc(-30:-80: 2.5cm and  2cm);
\node[draw, anchor = center, circle, fill = black, inner sep = 0pt, minimum width=3pt] (xi4) at (end) {};
\node[anchor = west] at (xi4) {\small $(\xi,\theta)\,(T_4)$};
\node[anchor = west] at (theta_extremum) {\small $(\xi,\theta)\,(T_2)$};
\end{tikzpicture}
\end{center}
\caption{A sketch of the argument for Proposition \ref{prop: position}(i). The black curve gives the evolution of $(\xi,\theta)$ up until the extremum of $\theta$. The dashed blue line describes the worst-case scenario for the evolution of $(\xi,\theta)\,(t)$ after this extremum. The red line, which crosses the line $\theta=\theta^*$ without any extrema of $\xi$, is an estimate of actual evolution starting on a certain point of the worst-case scenario.}
\label{fig: position(i)}
\end{figure}
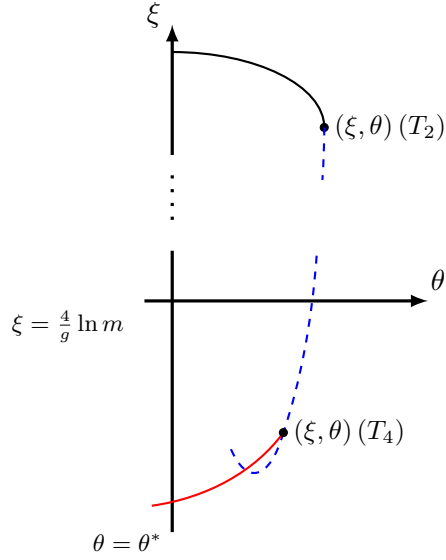
In what follows $\xi, \theta$ and $\alpha$ will denote the components of the solution of $(*)$ with initial condition $(\xi,\theta,\alpha)\,(0)=(\xi_0,\theta^*, \frac\pi2)$. The proof begins by establishing an auxilliary time $T_1$, at which $\frac{\xi'(T_1)}{\theta'(T_1)}=-1$.

\begin{lemma}\label{lemma: infty-welldef-1}
If $\xi_0$ is large enough then there is a time $T_1>0$ so that $\frac{\xi'(T_1)}{\theta'(T_1)}=-1$ while $\theta'(t)>0, \xi'(t)<0$ for all $t\in(0,T_1]$.
\end{lemma}
\begin{proof}
Note that one initially has $\frac d{dt}\cos\alpha\lvert_{t=0} <0$ whence one gets $\xi'(t)<0$ for small $t$. By Lemma \ref{lemma: r-turn} $\xi(t)$ then descends to $\frac g4 \ln m$ and does not have any extrema until after this value is reached, meaning $\cos\alpha(t)<0$ for all $t\in(0,t_{m}]$ and some $t_{m}\in\mtr$ at which $\xi(t_{m})=\frac g4\ln m$. For $\xi_0$ large enough there will then be some intermediate time $T_1<t_{m}$ for which $\frac d{d\theta}\xi(\theta) = \frac{\xi'(t)}{\theta'(t)}=-1$ holds, since either $\theta'(t)=0$ for some $t\in [0,t_{m}]$ or the graph $\xi(\theta)$ must descend from $\xi_0$ at $\theta^*$ to $\frac g4 \ln m$ at some value $\theta< \frac\pi2$. In the second case the mean-value theorem implies that the graph achieves slope $-1$ at some point.
\end{proof}

\begin{lemma}\label{lemma: infty-xi1-theta1}
There are constants $c_1,c_2, c_3\in\mtr_{>0}$ so that if $\xi_0$ is large enough one has
$$c_1 e^{-\frac4g\xi_0}≤\theta(T_1)-\theta^* ≤ c_2 e^{-\frac4g\xi_0}, \qquad \xi(T_1) \geq \xi_0+ c_3 e^{-\frac4g\xi_0}$$
\end{lemma}\begin{proof}
Before beginning with the proper analysis one notes that by well definedness one has $\theta(T_1) \leq \frac\pi2$, whence by the mean value theorem $\xi_0-\xi(T_1) ≤ \frac\pi2 - \theta^*$, some finite value bounded from above.

For the proof of this lemma it is more convenient to work with the following system of ODEs:
\begin{align*}
\xi'(t) &= \cos\alpha,\\ 
\theta'(t) &= \sin\alpha,\\
\alpha'(t) &= \sin\alpha\, (e^{\frac 4g\xi}-m) +2\cos\alpha\, H(\theta).
\end{align*}
which one can get from $(*)$ by rescaling time by the law $\dfrac{dt\subs{new}}{dt\subs{old}} = \sin(2\theta)$. To keep the number of superfluous parameters at a minimum we still use $T_1$ to denote the time at which $\frac{\xi'(T_1)}{\theta'(T_1)}=-1$ in this new ODE. One then gets for all $t\in[0,T_1]$:
$$\cos\alpha(t) \in [-\frac1{\sqrt2},0], \qquad \sin\alpha(t) \in [\frac1{\sqrt2},1].$$
Which implies:
$$\xi_0-\xi(T_1) \in [0,\frac1{\sqrt 2}T_1],\qquad \theta(T_1)-\theta^*\in [\frac1{\sqrt2}T_1,T_1].$$
We then procceed by bounding $T_1$ from above and below. Noting that for $t\in[0,T_1]$:
$$\frac1{\sqrt2}(e^{\frac4g\xi(T_1)} -m) ≤\alpha'(t) ≤ e^{\frac4g\xi_0} -\sqrt2 H(\theta(T_1)).$$
Integrating the left inequality from $0$ to $T_1$ yields:
$$\frac{T_1}{\sqrt2} (e^{\frac4g\xi(T_1)} -m) ≤\alpha(T_1)-\alpha(0)=\frac\pi4.$$
Recalling that $\xi_0-\xi(T_1) \leq \frac\pi2-\theta^*$ one gets $T_1≤ d_1 e^{-\frac4g\xi_0}$ for an appropriate constant $d_1$. This implies $\theta(T_1)-\theta^* ≤ c_2 e^{-\frac4g\xi_0}$ and $\xi_0-\xi(T_1)\leq c_3 e^{-\frac4g\xi_0}$ for appropriate $c_2, c_3$.

Combining $\theta(T_1)-\theta^*\leq c_2 e^{-\frac4g \xi_0}$ with $H(\theta^*)=0$ gives for $\xi_0$ large enough that $-\sqrt 2 H(\theta(T_1))  ≤ d_2  e^{-\frac4g\xi_0}$
for some other constant $d_2$. Integrating the other inequality for $\alpha'(t)$ from $0$ to $T_1$ then gives:
$$\frac\pi4\leq T_1 (e^{\frac4g\xi_0} + d_2 e^{-\frac4g\xi_0})\quad \implies\quad T_1 ≥ d_3 e^{-\frac4g\xi_0}$$
for another constant $d_3$, provided $\xi_0$ is large enough. This yields the final bound of the lemma, namely: $c_1e^{-\frac4g\xi_0}≤\theta(T_1)-\theta^*$.
\end{proof}
\begin{lemma}\label{lemma: infty-step2}
For $\xi_0$ large enough there is a time $T_2>T_1$ so that $\theta'(T_2)=0$, $\xi(T_2)>\xi_0-\frac1{\xi_0}$ while $\theta(t)\in\theta^*+(0,\frac1{\xi_0})$ and $\xi'(t)<0$ for all $t\in(0,T_2]$.
\end{lemma}
\begin{proof}
For all $t\in(0,\frac1{\xi_0})$ one has from $|\xi'(t)|\leq1$ that $\xi(t)>\xi_0-\frac1{\xi_0}$, so for $\xi_0$ large enough $\xi'(t)<0$ for such $t$. Further as long as $\theta>\theta^*$ and $\alpha\in [\frac{3\pi}4,\pi)$ one has for such $t$ that
$$\alpha'(t)= \sin\alpha\,\sin(2\theta)(e^{\frac 4g\xi}-m)+2\cos\alpha\, l(\theta)$$
is a sum of two positive terms and so $\alpha$ is increasing. Recalling that by definition $\alpha(T_1)=\frac{3\pi}4$ and assuming that $\alpha(t)<\pi$ (i.e. $\theta'(t)>0$) for all $t\in(T_1,\frac1{\xi_0})$ yields:
$$\alpha'(t) > A e^{\frac 4g(\xi_0-\frac1{\xi_0})} (\pi -\alpha).$$
Here $A>0$ is some constant. For this estimate we used that $\theta(t)$ is bounded away from $\{0,\frac\pi2\}$ for $t\in(0,\frac1{\xi_0})$, which follows from $|\theta'(t)|\leq1$. From the intermediate value theorem one then gets a $\widetilde t\in(T_1,\frac1{2\xi_0})$ so that
$$\alpha(\widetilde t) = \pi-A \exp\left(-e^{\frac 4g(\xi_0-\frac1{\xi_0})} \frac1{2\xi_0}\right).$$
For $\xi_0$ large enough one then finds $\alpha(\widetilde t) >\pi-e^{-\xi_0^2}$. Now $\theta(T_1)>\theta^*+c_1 e^{-\frac4g\xi_0}$ from Lemma \ref{lemma: infty-xi1-theta1}, so $l(\theta(T_1))\leq c_1l'(\theta^*) e^{-\frac4g \xi_0}$ for $\xi_0$ large enough. By assumption $\theta(\widetilde t)>\theta(T_1)$, so another estimate yields for $t\in(\widetilde t,\frac1{\xi_0})$:
$$\alpha'(t) > B\, e^{-\frac4g\xi_0}$$
where $B>0$ is some constant incorporating the $2\cos(\alpha)$ term (which is close to $-2$) and $c_1l'(\theta^*)$ (which is bounded away from $0$). This then yields
$$\alpha(\widetilde t+\frac1{2\xi_0}) >\pi - e^{-\xi_0^2}+B\frac{e^{-\frac 4g\xi_0}}{2\xi_0}$$
which is larger than $\pi$, contradicting our assumption that $\alpha(t)<\pi$ for all $t\in(T_1,\frac1{\xi_0})$. The lemma then follows.
\end{proof}

We will now prove that $\xi_0$ is of type 1 by contradiction.
\begin{lemma}\label{lemma: infty-xi0-finish}
If $\xi_0$ is large enough then it is of type 1.
\end{lemma}
\begin{proof}
If $\xi_0$ is large enough by Lemma \ref{lemma: infty-step2} there is a $T_2>0$ for which $\theta'(T_2)=0$, which corresponds to a local maximum of $\theta$. This means $\theta$ reaches $\theta^*$ in finite time by Lemma \ref{lemma: theta-turn} and so $\xi_0$ is not type 3. We now assume it is not type 1, so it must be type 2.

Hence $\xi(t)$ has an extremum before $\theta(t)$ reaches $\theta^*$. We let $T_3>T_2$ denote the time of this extremum and note first that $\theta(t)\in(\theta^*,\theta^*+\frac1{\xi_0})$ for all $t\in[T_2,T_3)$ (since $\theta(T_2)\leq\theta^*+\frac1{\xi_0}$ is a maximum and any further extremum of $\theta$ must take place behind the line $\theta=\theta^*$).

Next one sees that $\xi(T_3)<\frac g4\ln m$, since the extremum must be a minimum. With $\xi(T_2)\geq\xi_0+\frac1{\xi_0}$ and $|\xi'(t)|\leq1$ one gets that $T_3-T_2$ will become arbitrarily large as $\xi_0$ grows. And so, for $\xi_0$ large enough, one sees that $\xi(t)\leq\frac g4\ln m+1$, $|\theta(t)-\theta^*|\leq 1$ for all $t\in[T_3-1,T_3]$. Note that $\alpha'(t)$ admits a bound if $\xi,\theta$ are in this region.

Now $\xi'(T_3) = 0$ implies $\cos\alpha(T_3)=0$ and so $\sin\alpha(T_3)=-1$, which in turn gives $\theta'(T_3)=-\sin(2\theta)\leq-1+O(\frac1{\xi_0^2})$. From $(*)$ one sees directly that $|\theta''(t)|\leq |\alpha'(t)|+1$, and so the bound on $\alpha'(t)$ give a finite $b>0$ (independent of $\xi_0$) so that $\theta'(t)<-\frac12$ for all $t\in [T_3-b,T_3]$.

But if $\xi_0$ is large enough one notes that $\theta(t)\in [\theta^*,\theta^*+\frac1{\xi_0})$ and $\theta'(t)\leq-\frac12$ cannot both simultaneously hold for all $t\in[T_3-b,T_3]$. This contradiction shows that $\xi_0$ cannot be type 2, hence (since it is also not type 3) it is type 2.
\end{proof}

\subsection{Proof of Proposition \ref{prop: position} (ii)}\label{subsec: (ii)}

We consider the solution curve with initial condition $\xi_0=\frac g4\ln m +\epsilon$ and show that this is not of type 1 for $\epsilon$ sufficiently small. To do this we assume that it is of type $1$ - only to later arrive at a contradiction. If it were of type 1 then there is a $T>0$ with $\theta(T) = \theta^*$ and $\xi'(t)<0$ for all $t\in (0,T)$. Since $\theta$ can only have maxima when $\theta>\theta^*$, we find that the trajectory $\{(\theta,\xi)\,(t)\mid t\in[0,T]\}$ must be the union of two graphs of $\xi$ over $\theta$. The maximum of $\theta$ occurs at the point denoted by $(\xi_2,\theta_2)$ in Figure \ref{fig: position(ii)}.

In the upper graph one has that the slope $\frac{d\xi}{d\theta}$ starts at $0$ and must go to $-\infty$ (which occurs when $\theta'(t)=0$). Along the way $\xi'$ has been negative and one verifies that $\xi_2$ has decreased to a value far enough below $\frac g4\ln m$ (c.f. Lemma \ref{lemma: m-xi2}). When we then switch to the lower graph the $e^{\frac4g\xi}-m$ term in the ODE for $\frac{d^2\xi}{d\theta^2}(\theta)$ will be large enough to push $\frac{d\xi}{d\theta}$ over the value $0$ before $\theta$ reaches $\theta^*$, contradicting the assumption that $\xi_0$ was type 1.

Note that this does not prove that $\frac g4\ln m+\epsilon$ is of type 2, because the proof by contradiction assumes that $\theta(t)$ has a maximum.

In what follows $\xi, \theta$ and $\alpha$ will denote the components of the solution of $(*)$ with initial condition $(\xi,\theta,\alpha)\,(0)=(\frac4g \ln m+\epsilon,\theta^*, \frac\pi2)$. The proof begins by investigating an auxilliary value $\theta_1$, which is defined to be the least (and for small $\epsilon$ only) value of $\theta$ for which one has $\frac{d\xi}{d\theta}(\theta_1)=-1$ in the upper graph.

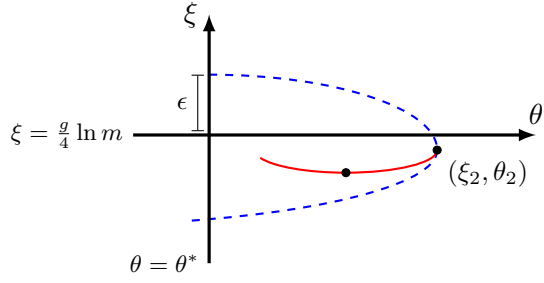
\begin{figure}\begin{center}
\begin{tikzpicture}
\draw[blue, thick, dashed] (0,3) arc (90:0:3 cm and 1cm);
\draw[blue,thick, dashed] (3,2) arc (0:-70:5 cm and 1 cm);
\draw[|-|] (-0.15, 2.25) -- (-0.15, 3) node [midway, anchor = east] {\small$\epsilon$};
\draw[-latex, very thick] (0,.5) -- (0,3.8) node[anchor = east] {$\xi$};
\node[anchor = east] at (-1,2.2) {\footnotesize $\xi=\frac g4\ln m$};
\node[anchor = east] at (0,0.5) {\footnotesize $\theta=\theta^*$};
\draw[-latex, very thick] (-1, 2.2) -- (4.3, 2.2) node[anchor = south] {$\theta$};
\draw[red,thick] (3,2) arc(0:-160: 1.2 cm and .3cm);
\node[draw, anchor = center, circle, fill = black, inner sep = 0pt, minimum width=3pt] (P) at (3,2) {};
\node[draw, anchor = center, circle, fill = black, inner sep = 0pt, minimum width=3pt] (Q) at (1.8,1.7) {};
\node[anchor = north west] at (P) {\small $(\xi_2,\theta_2)$};
\end{tikzpicture}
\caption{The figure sketches the argument for Proposition \ref{prop: position}(ii). The dashed blue line denotes the form that $(\xi,\theta)\,(t)$ must be if the initial condition were type 1. The red line, which has an extremum of $\xi$, is an estimate of the actual evolution starting at $(\xi_2,\theta_2)$.}
\label{fig: position(ii)}
\end{center}
\end{figure}

\begin{lemma}\label{lemma: m-welldef}
If $\xi_0=\frac g4 \ln m+\epsilon$ is type 1, then there are $T_2(\epsilon)>T_1(\epsilon)>0$ so that $\theta'(T_2)=0$ and $\frac{\xi'(T_1)}{\theta'(T_1)}=-1$, while $\xi'(t)<0$ and $\theta'(t)>0$ for all $t\in(0,T_2)$.\end{lemma}
\begin{proof}
Assuming that $\xi_0$ is of type 1 means that there is a time $T>0$ for which $\theta(T)=\theta^*$ and $\xi'(t)\neq0, \theta(t)\neq\theta^*$ for all $t\in(0,T)$. Since $\theta'(0) =\sin(2\theta^*)>0$ one finds that $\theta'(t)>0$ for small $t$, whence $\theta(t)$ must go through an extremum before it can go back to $\theta^*$ and there is a $T_2<T$ so that $\theta'(T_2)=0$. On the other hand one has $\xi''(0) = -\sin(2\theta^*) m (e^{\frac4g\epsilon} -1)<0$, whence $\xi'(t)<0$ for all $t\in(0,T]$, in particular for all $t\in(0,T_2)$.

This means that $\frac{\xi'(t)}{\theta'(t)}$ is $0$ at $t=0$ and diverges to $-\infty$ at $t=T_2$. There must then be a $T_1$ so that $\frac{\xi'(T_1)}{\theta'(T_1)}=-1$.
\end{proof}
\begin{lemma}\label{lemma: m-theta1pi}
For $\epsilon$ small enough there is only one pair $(T_1, T_2)$ satisfying the conditions of Lemma \ref{lemma: m-welldef} and $\theta(T_1)\to\frac\pi2$ as $\epsilon\to0$.\end{lemma}
\begin{proof}
The time $T_2$ is obviously unique.

On the other hand the initial condition $\epsilon=0$ has as solution the line $(\theta, \xi, \alpha) = (\arctan(\tan(\theta^*)e^{2t}),\frac g4\ln m, \frac\pi2)$. So one finds that as $\epsilon\to0$ the solution and its derivatives converge uniformly on compacta to the above curve, in particular for on any finite interval $[0,T]$ one can make $\frac{\xi'(t)}{\theta'(t)}$ arbitrarily small for all $t\in[0,T]$ (by taking $\epsilon$ small), while $\theta(T)$ is arbitrarily close to $\frac\pi2$ (by taking $T$ large and then $\epsilon$ small). This means that as $\epsilon\to0$ one must have $\frac\pi2-\theta (T_1(\epsilon))\to0$, where $T_1$ is any of the times satisfying Lemma \ref{lemma: m-welldef}.

Looking however at the graph ODE $(**)$
$$\frac{d^2\xi}{d\theta^2}=-(1+(\frac{d\xi}{d\theta})^2)(e^{\frac4g\xi}-m +2H(\theta)\frac{d\xi}{d\theta})$$
it follows if $\frac{d\xi}{d\theta}\leq-1$ and $\theta$ is close enough to $\frac\pi2$ while $\xi$ is not too large that then $\frac{d^2\xi}{d\theta^2}<0$, since $H(\theta)\to-\infty$ as $\theta\to\frac\pi2$. For $\epsilon$ small enough one gets that for any pair $(T_1,T_2)$ satisfying Lemma \ref{lemma: m-welldef} there is no pair $(T_1',T_2)$ satisfying the lemma with $T_1'>T_1$.
\end{proof}

This allows us to introduce the following notation:
$$\theta_1\defeq \theta(T_1),\qquad \theta_2\defeq\theta(T_2), \qquad \xi_2 \defeq \xi(T_2).$$
\begin{lemma}\label{lemma: m-xi1}
For $\epsilon$ small enough one has that $\xi(T_1)\leq \frac g4\ln m$.
\end{lemma}
\begin{proof}
We assume that $\xi(T_1)>\frac g4\ln m$ and get a contradiction. The first step is to show that this assumption leads to constants $c_1, c_2>0$ (independent of $\epsilon$) so that for $\epsilon$ small enough one has:
\begin{equation}c_1 \epsilon < (\frac\pi2-\theta_1)^{m_2}<c_2\epsilon.\label{eq: epsilon-wedge}\end{equation}
Since $\xi'(t)\leq0$ for all $t\in(0,T_2)$ we may assume $\xi(t)>\frac g4\ln m$ for such $t$. Then from the graph ODE
$$\frac{d^2\xi}{d\theta^2}=-(1+(\frac{d\xi}{d\theta})^2)(e^{\frac4g\xi}-m +2H(\theta)\frac{d\xi}{d\theta})$$
one gets that $\frac{d^2\xi}{d\theta^2}<0$ for all $\theta\in(\theta^*,\theta_1)$ (recall that $H(\theta)<0$ for $\theta>\theta^*$) and then $\frac{d\xi}{d\theta}$ is strictly decreasing in this interval. Monotonicity of $\frac{d\xi}{d\theta}$ and mean value theorem then imply
\begin{equation}
\frac{d\xi}{d\theta}\left(\frac{\theta_1-\theta^*}2\right) > -\epsilon\, \frac{\theta_1-\theta^*}2,\qquad \xi\left(\frac{\theta_1-\theta^*}2\right) > \frac g4 \ln m+\frac\epsilon2.\label{eq: xip-bound-xi1}
\end{equation}

Combining the second inequality with the graph ODE yields (recall that $H(\theta)<0$ if $\theta>\theta^*$) additionally the bound $\frac{d\xi}{d\theta}\left(\frac{\theta_1-\theta^*}2\right) < -\frac{m\epsilon}g(\theta_1-\theta^*)+O(\epsilon^2)$. Using that $\theta_1\to\frac\pi2$ as $\epsilon\to0$ then gives constants $\widetilde c_1, \widetilde c_2$ so that
\begin{equation}\widetilde c_1\epsilon < \left|\frac{d\xi}{d\theta}\left(\frac{\theta_1-\theta^*}2\right)\right|<\widetilde c_2\epsilon.\label{eq: midpoint-epsilon}\end{equation}
Rewriting the graph ODE as
\begin{equation}
\frac{\frac{d^2\xi}{d\theta^2}}{\frac{d\xi}{d\theta}(1+(\frac{d\xi}{d\theta})^2)} = - \frac{e^{\frac4g \xi}-m}{\frac{d\xi}{d\theta}}-2H(\theta),\label{eq: log-xigr}
\end{equation}
one notes that the first term on the right-hand side is $O(1)$ in the interval $(\frac{\theta_1-\theta^*}2,\theta_1)$ by (\ref{eq: midpoint-epsilon}) and monotonicity of $\frac{d\xi}{d\theta}$, and further that the left-hand side has $-\ln(\frac{|\frac{d\xi}{d\theta}|}{\sqrt{1+(\frac{d\xi}{d\theta})^2}})$ as an anti-derivative. Integrating (\ref{eq: log-xigr}) over $(\frac{\theta_1-\theta^*}2,\theta_1)$ then gives
$$\ln(\left|\frac{d\xi}{d\theta}\left(\frac{\theta_1-\theta^*}2\right)\right|) + O(1) =  m_2\ln(\cos(\theta_1))+O(1).$$
(For the left-hand side: From $-\epsilon\frac{\theta_1-\theta^*}2<\frac{d\xi}{d\theta}\leq-1$ for $\theta\in(\frac{\theta_1-\theta^*}2,\theta_1)$ the $1+(\frac{d\xi}{d\theta})^2$ term in the anti-derivative is absorbed into the $O(1)$, similarly from $\frac{d\xi}{d\theta}(\theta_1)=-1$ only the lower boundary of the integral has a contribution.)

The approximation $\cos(\theta) = \frac\pi2-\theta + O((\frac\pi2-\theta)^3)$ then gives the bounds (\ref{eq: epsilon-wedge}) from (\ref{eq: midpoint-epsilon}).

For the next step note that $H(\theta)-({\frac{d\xi}{d\theta}(\theta)})^{-1}$ becomes $+\infty$ as $\theta\to\theta^*$ and $H(\theta_1)+1$ as $\theta\to\theta_1$, which for $\epsilon$ small enough will be negative. Hence for small enough $\epsilon$ there is a $\theta_0\in(\theta^*, \theta_1)$ so that $H(\theta_0)=({\frac{d\xi}{d\theta}(\theta_0)})^{-1}$. In the same way as Lemma \ref{lemma: m-theta1pi} one shows that $\theta_0\to\frac\pi2$ as $\epsilon\to0$.

We let
$$q\defeq\frac{\pi/2-\theta_0}{\pi/2-\theta_1}$$
and show next that there is a constant $c_3$ so that $q^{m_2}\geq \frac{c_3}{\pi/2-\theta_0}$, and hence that $q$ grows unboundedly as $\epsilon\to0$.

To show this we again integrate (\ref{eq: log-xigr}), this time from $\theta_0$ to $\theta\leq \theta_1$. The result is:
$$\ln(|\frac{d\xi}{d\theta}\left(\theta_0\right)|)-\ln(|\frac{d\xi}{d\theta}(\theta)|) + O(1)=m_2\ln\left(\frac{\cos(\theta)}{\cos(\theta_0)}\right)+O(1).$$
(As before the $1+(\frac{d\xi}{d\theta})^2$ term on the left-hand side is absorbed into the $O(1)$, while the $\frac{e^{\frac4g\xi}-m}{\frac{d\xi}{d\theta}}$ term on the right is also $O(1)$ by (\ref{eq: midpoint-epsilon}) and monotonicity of $\frac{d\xi}{d\theta}$.)

Noting that $\frac{d\xi}{d\theta}(\theta_0)=\frac1{H(\theta_0)}\approx -m_2 (\frac\pi2-\theta_0) + O((\frac\pi2-\theta_0)^2)$ and again using the expansion of cosine close to $\frac\pi2$ implies the existence of $c_3, c_4>0$ so that:
\begin{equation}
c_4\left|\frac{d\xi}{d\theta}(\theta)\right|\geq \left(\frac{\pi/2-\theta_0}{\pi/2-\theta}\right)^{m_2} (\frac\pi2-\theta_0)\geq c_3 \left|\frac{d\xi}{d\theta}(\theta)\right|.\label{eq: xip-bound}
\end{equation}
Taking $\theta=\theta_1$ shows $q^{m_2}\geq \frac{c_3}{\pi/2-\theta_0}$, in particular $q\to\infty$ as $\epsilon\to0$. On the other hand, if one integrates $\frac{d\xi}{d\theta}(\theta)$ over $(\theta_0,\theta_1)$ one finds by (\ref{eq: xip-bound}) that
$$c_4|\xi(\theta_1)-\xi(\theta_0)|\geq \begin{cases}(\frac\pi2-\theta_0)^2\ln q& m_2=1\\ \frac1{m_2-1}(\frac\pi2-\theta_0)^2(q^{m_2-1}-1)& m_2>1\end{cases}.$$
Noting that $(\frac\pi2-\theta_0)^2 =\frac{(\pi/2-\theta_1)^{m_2}}{(\pi/2-\theta_0)^{m_2}}q^{m_2} (\frac\pi2-\theta_0)^2\geq \epsilon\, c_1c_3 (\frac\pi2-\theta_0)^{1-m_2}$ and using unboundedness of $q$ one sees that $|\xi(\theta_1)-\xi(\theta_0)|$ will be larger than $\epsilon$, contradicting our assumption that $\xi(\theta_1)\geq \frac g4\ln(m)$.
\end{proof}

\begin{lemma}\label{lemma: m-xi2}
For $\theta_1$ sufficiently close to $\frac\pi2$ one has that
$$\xi(T_2) \leq\frac g4\ln(m) -  \left(2^{\frac1{2m_2}}-1\right)(\frac\pi2 -\theta_2) + O((\frac\pi2-\theta_2)^2).$$
\end{lemma}
\begin{proof}
By Lemma \ref{lemma: m-xi1} one has that $\xi(T_1) \leq \frac g4\ln(m)$. Additionally for $\epsilon$ small enough $\frac {d^2\xi}{d\theta^2}(\theta)<0$ for all $\theta\in(\theta_1,\theta_2)$, whence $\xi(T_2) \leq \xi(T_1) - (\theta_2-\theta_1)$ and the statement reduces to checking $\theta_2-\theta_1 = \left(2^{\frac1{2m_2}}-1\right)(\frac\pi2-\theta_2) +O((\frac\pi2-\theta_2)^2)$.

From the graph ODE $(**)$ one recovers:
$$\frac{\frac{d^2\xi}{d\theta^2}}{\frac{d\xi}{d\theta}(1+(\frac{d\xi}{d\theta})^2)} = - \frac{e^{\frac4g \xi}-m}{\frac{d\xi}{d\theta}}-2H(\theta).$$
The absolute value of the first summand on the right-hand side is bounded by $m$ over the interval $(\theta_1,\theta_2)$, so integrating the equation from $\theta_1$ to $\theta_2$ gives:
$$\frac12\ln(2) = O(\frac\pi2-\theta_1)-m_2\ln\left(\frac{\cos(\theta_2)}{\cos(\theta_1)}\right)-m_1\ln\left(\frac{\sin(\theta_2)}{\sin(\theta_1)}\right).$$
As $\theta$ becomes arbitrarily close to $\frac\pi2$ we use that $\cos(\theta) = \frac\pi2-\theta+O((\frac\pi2-\theta)^3)$, $\sin(\theta)=1+O((\frac\pi2-\theta)^2)$ to get:
$$\frac12\ln(2)=-m_2 \ln\left(\frac{\frac\pi2-\theta_2}{\frac\pi2 -\theta_1}\right) + O(\frac\pi2-\theta_1).$$
Together with some arithmetic this implies the statement about $\theta_2-\theta_1$.
\end{proof}

\begin{lemma}
If $\epsilon$ is small enough then $\xi_0=\frac g4\ln m+\epsilon$ is not type 1.
\end{lemma}
\begin{proof}
As noted before the assumption that $\xi_0=\frac g4\ln(m)+\epsilon$ is type 1 leads to $(\xi,\theta)$ being the union of two graphs of $\xi$ over $\theta$. In the previous lemmas we investigated the upper graph and found that it ends at the turning point $(\xi_2,\theta_2)$.

The lower graph is then determined by the graph ODE and the initial conditions $\xi(\theta_2)=\xi_2$, $\lim_{\theta\to\theta_2^-}\frac{d\xi}{d\theta}(\theta)=+\infty$. The assumption that $\xi_0$ is type 1 necessitates that for all $\theta\in(\theta^*,\theta_2)$ one has $\frac{d\xi}{d\theta}(\theta)>0$ for the lower graph. In particular if we integrate
\begin{equation}
\frac{\frac{d^2\xi}{d\theta^2}}{\frac{d\xi}{d\theta}(1+(\frac{d\xi}{d\theta})^2)} = - \frac{e^{\frac4g \xi}-m}{\frac{d\xi}{d\theta}}-2H(\theta)\label{eq: log-xigr2}
\end{equation}
from $\frac{\theta^*+\theta_2}2$ to $\theta_2$ one gets
$$-\ln\left(\frac{\frac{d\xi}{d\theta}(\frac{\theta^*+\theta_2}2)}{\sqrt{1+(\frac{d\xi}{d\theta})^2}}\right) \geq -m_2 \ln(\frac\pi2-\theta_2)+O(1),$$
where we used $e^{\frac 4g\xi(\theta)}-m \leq e^{\frac 4g\xi_2}-m \leq 0$. Inverting the above inequality gives the existence of a $c_1>0$ so that:
\begin{equation}
\frac{d\xi}{d\theta}\left(\frac{\theta^*+\theta_2}2\right) \leq \frac{c_1(\frac\pi2-\theta_2)^{m_2}}{\sqrt{1-c_1^2(\frac\pi2-\theta_2)^{2m_2}}}\leq 2c_1 (\frac\pi2-\theta_2)^{m_2}.\label{eq: m-xip-bound2}
\end{equation}
For $m_2>1$ this implies the lemma, since for all $\theta\in(\theta^*,\theta_2)$ one has
$$\frac{d^2\xi}{d\theta^2} (\theta)> -(e^{\frac4g\xi(\theta)}-m) \geq c_2(\frac\pi2-\theta_2)$$
with $c_2>0$ some constant by Lemma \ref{lemma: m-xi2}, and then
$$\frac{d\xi}{d\theta}(\theta)\leq 2c_1(\frac\pi2-\theta_2)^{m_2}- c_2(\frac\pi2-\theta_2)\,(\frac{\theta^*+\theta_2}2-\theta)$$
for $\theta\in(\theta^*,\frac{\theta^*+\theta_2}2)$ and one gets $\frac{d\xi}{d\theta}(\theta)=0$ for one such $\theta$.

For $m_2=1$ one gets first from (\ref{eq: m-xip-bound2}) that $\frac{d\xi}{d\theta}$ takes on all values in $(\frac{c_1}2(\frac\pi2-\theta_2), \infty)$ as $\theta$ varies from $\frac{\theta^*+\theta_2}2$ to $\theta_2$. In particular for $\theta_2$ close enough to $\frac\pi2$ there exists a $\theta_3<\theta_2$ so that $\frac{d\xi}{d\theta}(\theta_3)=-\frac1{H(\theta_3)}$.

By integrating (\ref{eq: log-xigr2}) one shows that $\theta_3$ gets arbitrarily close to $\frac\pi2$ if $\frac\pi2-\theta_2$ is small enough. We carry this out explicitly:

Integrate (\ref{eq: log-xigr2}) from $\theta_3$ to $\theta_2$, the left-hand side evaluates to $\frac{-1}2\ln(1+H(\theta_3)^2)$, which is negative and remains bounded as $\theta_2\to\frac\pi2$ unless $\theta_3$ also gets close to $\frac\pi2$ (since otherwise $H(\theta_3)$ is bounded). For the right-hand side one first notes that the integral over $-\frac{e^{\frac4g \xi}-m}{\frac{d\xi}{d\theta}}$ from $\theta_3$ to $\theta_2$ is positive. The other term on the right-hand side however yields $-\ln(\frac{\pi/2-\theta_2}{\pi/2-\theta_3})+O(1)$ which is, crucially, negative and unbounded unless $\theta_3\to\frac\pi2$ together with $\theta_2$. So a situation where $\theta_2\to\frac\pi2$ but $\theta_3\not\to\frac\pi2$ is impossible.

In fact this integral shows that if $q\defeq\frac{\pi/2-\theta_3}{\pi/2-\theta_2}$ that $q$ grows unboundedly as $\theta_2$ approaches $\frac\pi2$. Finally if we integrate (\ref{eq: log-xigr2}) from $\theta\geq \theta_3$ to $\theta_2$ one gets:
$$-\ln\left(\frac{\frac{d\xi}{d\theta}(\theta)}{\sqrt{1+(\frac{d\xi}{d\theta}(\theta))^2}}\right) +O(1) = -\ln(\frac{\pi/2-\theta_3}{\pi/2-\theta})+O(\frac\pi2-\theta).$$
(Here the $-\frac{e^{\frac4g\xi}-m}{\frac{d\xi}{d\theta}}$ term is bounded by $-H(\theta_3)m = \frac m{\pi/2-\theta_3}+O(1)$. Its contribution to the integral then an $O(1)$ term, since $\frac{\theta_2-\theta}{\pi/2-\theta_3}\leq1$.)

Inverting this expression gives a constant $c_3$ so that $\frac{d\xi}{d\theta}(\theta)\geq c_3 \frac{\pi/2-\theta_2}{\pi/2-\theta}$, and then integrating this from $\theta_3$ to $\theta_2$ gives:
$$\xi(\theta_2)-\xi(\theta_3)\geq c_3(\frac\pi2-\theta_2)\ln q$$
Plugging this into $(**)$ then implies for any $\theta\in(\theta^*,\frac{\theta^*+\theta_2}2)$ that:
$$\frac{d^2\xi}{d\theta^2}(\theta)\geq m-e^{\frac4g \xi(\theta_3)}\geq m(1-e^{-\frac4g c_3(\frac\pi2-\theta_2)\ln q}).$$
Together with (\ref{eq: m-xip-bound2}), which states $\frac{d\xi}{d\theta}(\frac{\theta^*+\theta_2}2)\leq 2c_1 (\frac\pi2-\theta_2)$, and unboundedness of $q$ as $\theta_2\to\frac\pi2$ one recovers $\frac{d\xi}{d\theta}(\theta)=0$ for some $\theta\in(\theta^*,\frac{\theta^*+\theta_2}2)$, provided $\frac\pi2-\theta_2$ is small enough. 

So also in the case $m_2=1$ we get a contradiction to the assumption that $\frac g4\ln m+\epsilon$ was type 1 for $\epsilon$ small enough.
\end{proof}

\subsection{Proof of Proposition \ref{prop: position} (iii)}\label{subsec: (iii)}

\begin{lemma}
If $\xi_0^*$ is of type 3 then:
\begin{enumerate}[label=(\roman*)]
\item There is a $\delta>0$ and a $T>0$ so that $\xi_{\xi_0^*}(t)<\frac g4\ln(m)-\delta$ for all $t>T$.
\item $\theta'_{\xi_0^*}(t)>0$ for all $t>0$.
\item $\lim_{t\to\infty}\theta_{\xi_0^*}(t)=\frac\pi2$.
\end{enumerate}
\end{lemma}
\begin{proof}
Since $\xi_0^*>\frac g4\ln m$ and $\xi_{\xi_0^*}''(0) = - \sin(2\theta^*)(e^{\frac 4g\xi_0^*}-m)<0$, $\xi_{\xi_0^*}'(0)=0$ one has that $\xi_{\xi_0^*}'(t)$ becomes negative for small times $t$. Then by Lemma \ref{lemma: r-turn} it must reach $\frac g4\ln(m)$ in finite time. Since $\xi_{\xi_0^*}'(t)<0$ for all $t>0$ this yields part (i).

One also has that $\theta_{\xi_0^*}'(t)\neq 0$ for all $t>0$, as otherwise Lemma \ref{lemma: theta-turn} implies that $\theta_{\xi_0^*}=\theta^*$ in finite time. This gives part (ii).

$\theta_{\xi_0^*}$ is then monotonous in $t$ and bounded by $\frac\pi2$, whence it converges. Since $\xi_0^*$ is bounded above one has that all derivatives of the parameters are bounded and hence $\theta'_{\xi_0^*}(t)=\sin\alpha\,\sin(2\theta)\to0$. It is clear that $\sin\alpha$ cannot converge to $0$ as $\alpha'(t)$ would then be asymptotically equal to $-2l(\theta)$, which does not converge to $0$. Hence $\theta\to\frac\pi2$, giving part (iii).
\end{proof}

\begin{lemma}
$\xi_0^*$ is not of type 3.
\end{lemma}
\begin{proof}
Let $(\xi_\epsilon(t),\theta_\epsilon(t))$ denote the solution to $(*)$ with initial condition $(\xi, \theta, \alpha)\ (t=0)=(\xi_0^*+\epsilon, \theta^*, \frac\pi2)$ where $\epsilon>0$. Note that $\xi_0^*+\epsilon$ is type 1 by the definition of $\xi_0^*$, in particular there is a $T_1(\epsilon)$ so that $\theta_\epsilon(t)$ has a local maximum and $\xi_\epsilon'(T_1)<0$ for all $t\in(0,T_1]$.

Assuming that $\xi_0^*$ is of type 3 and using that $\xi_\epsilon$ and $\theta_\epsilon$ converge uniformly on compacta to $\xi_{\epsilon=0}$ and $\theta_{\epsilon=0}$ as $\epsilon\to0$ one finds that $\theta_\epsilon(T_1)\to \frac\pi2$ as $\epsilon\to0$. This also implies $T_1(\epsilon)\to\infty$ for $\epsilon\to0$, giving for $\epsilon$ small enough that one has $\xi_\epsilon(T_1)<\frac g4\ln(m)-\delta$.

After the extremum at $\theta_\epsilon(T_1)$ one has that $\xi_\epsilon$ becomes a graph over $\theta$ because no more extrema of $\theta$ are possible until after $\theta=\theta^*$. As in Lemma \ref{lemma: m-xi1} the graph ODE for $\xi$ yields:
$$\frac{\frac{d^2\xi_\epsilon}{d\theta^2}}{\frac{d\xi_\epsilon}{d\theta}(1+(\frac{d\xi}{d\theta})^2)} = -\frac{e^{\frac 4g\xi_\epsilon}-m}{\frac{d\xi_\epsilon}{d\theta}}-2H(\theta).$$
Integrating this for the lower graph from $x\defeq\frac{\theta^*+\pi/2}2$ to $\theta_\epsilon(T_1)$ one gets
$$-\ln\left(\frac{d\xi_\epsilon/d\theta}{\sqrt{1+(\frac{d\xi_\epsilon}{d\theta})^2}}(x)\right) \geq -m_2 \ln(\frac\pi2-\theta_\epsilon(T_1))+O(1)$$
where $-\frac{e^{\frac 4g\xi}-m}{\frac{d\xi}{d\theta}}\geq 0$ and $\frac{d\xi}{d\theta}(\theta(T_1))=+\infty$ were used.

As $\epsilon\to0$ this implies that $\frac{d\xi_\epsilon}{d\theta}(x)$ becomes arbitrarily small, in particular we may assume it to be smaller than $\frac{4m}g \frac12(x-\theta^*)\delta$. Remarking however that if $\xi<\frac g4\ln(m)-\delta$ then $-(e^{\frac 4g\xi}-m) > \frac4g m \delta$, one gets:
$$\frac{d^2\xi_\epsilon}{d^2\theta}(\theta) > \frac4g m \delta$$
for all $\theta\in(\theta^*, x)$. With $\frac{d\xi_\epsilon}{d\theta}(x)<\frac{4m}g\frac12(x-\theta^*)\delta$ one immediately gets $\frac{d\xi_\epsilon}{d\theta}(\frac{\theta^*+x}2) <0$, contradicting that $\xi_\epsilon$ is type 1.
\end{proof}

\bibliographystyle{amsplain}
\bibliography{my}

\end{document}